\documentclass[leqno]{siamart190516}
\usepackage{graphics,graphicx,epstopdf,url}

\usepackage{longtable}
\usepackage[tight]{subfigure}
\usepackage{float}
\usepackage{makecell}
\usepackage{comment}
\usepackage{mathtools}
\usepackage{amsxtra, amsfonts,amscd,amssymb,epsf}
\usepackage[algo2e,linesnumbered, vlined,ruled]{algorithm2e}
\numberwithin{equation}{section}

\newtheorem{assumption}[theorem]{Assumption}

\newtheorem{Theorem}{Theorem}[section]

\newtheorem{Remark}[theorem]{Remark}

\newtheorem*{proof sketch}{Proof Sketch}

\newcommand{\R}{\mathbb{R}}

\newcommand{\E}{\mathbb{E}}
\newcommand{\vect}{\mathrm{vec}}
\newcommand{\proj}{\mathcal{P}}

\newcommand{\Mcal}{\mathcal{M}}

\newcommand{\mat}{{\mathrm{mat}}}

\newcommand{\iprod}[2]{\left \langle #1, #2 \right \rangle }

\newcommand{\be}{\begin{equation}}
	\newcommand{\ee}{\end{equation}}
\newcommand{\bee}{\begin{equation*}}
	\newcommand{\eee}{\end{equation*}}
\newcommand{\bea}{\begin{eqnarray}}
	\newcommand{\eea}{\end{eqnarray}}
\newcommand{\beaa}{\begin{eqnarray*}}
	\newcommand{\eeaa}{\end{eqnarray*}}

\usepackage{xcolor}

\newcommand{\hess}{\mathrm{Hess\!\;}}
\newcommand{\grad}{\mathrm{grad\!\;}}

\DeclareMathOperator*{\argmin}{arg\,min}

\begin{document}

\title{Riemannian Natural Gradient Methods}
% Short title for running heads:
%\shorttitle{Preconditioning by inverting the Laplacian}

\author{Jiang Hu\thanks{Department of Systems Engineering and Engineering Management,
		 The Chinese University of Hong Kong, Shatin, NT, Hong Kong
		(\email{hujiangopt@gmail.com}).}	
	\and Ruicheng Ao\thanks{School of Mathematical Sciences, Peking University, China (\email{archer\_arc@pku.edu.cn}).}
	\and Anthony Man-Cho So\thanks{Department of Systems Engineering and Engineering Management, The Chinese University of Hong Kong, Shatin, NT, Hong Kong
		(\email{manchoso@se.cuhk.edu.hk}).}
	\and Minghan Yang\thanks{Beijing International Center for Mathematical Research, Peking University, China (\email{yangminghan@pku.edu.cn}).}
	\and Zaiwen Wen\thanks{Beijing International Center for Mathematical Research, Center for Data Science and College of Engineering, Peking University, Beijing, China
		(\email{wenzw@pku.edu.cn}).}}

\maketitle

\begin{abstract}
{This paper studies large-scale optimization problems on Riemannian manifolds whose objective function is a finite sum of negative log-probability losses. Such problems arise in various machine learning and signal processing applications. By introducing the notion of Fisher information matrix in the manifold setting, we propose a novel Riemannian natural gradient method, which can be viewed as a natural extension of the natural gradient method from the Euclidean setting to the manifold setting. We establish the almost-sure global convergence of our proposed method under standard assumptions. Moreover, we show that if the loss function satisfies certain convexity and smoothness conditions and the input-output map satisfies a Riemannian Jacobian stability condition, then our proposed method enjoys a local linear---or, under the Lipschitz continuity of the Riemannian Jacobian of the input-output map, even quadratic---rate of convergence. We then prove that the Riemannian Jacobian stability condition will be satisfied by a two-layer fully connected neural network with batch normalization with high probability, provided that the width of the network is sufficiently large. This demonstrates the practical relevance of our convergence rate result. Numerical experiments on applications arising from machine learning demonstrate the advantages of the proposed method over state-of-the-art ones.} 
%\revise{This is the first work on designing Riemannian natural gradient method for manifold optimization problems.}
\end{abstract}


\section{Introduction}
Manifold constrained learning problems are ubiquitous in machine learning, signal processing, and deep learning. 
%For example,  the parameters before the batch normalization (BN) \cite{ioffe2015batch}, layer
%normalization (LN) \cite{ba2016layer} layers in deep neural network lie on the product manifold of
%Grassmannian or Stiefel manifolds. 
In this paper, we focus on manifold optimization problems of the form
\be \label{prob}
\begin{aligned}
	\min_{\Theta \in \Mcal} \quad & \Psi(\Theta) := -\frac{1}{|\mathcal{S}|} \sum_{(x,y) \in \mathcal{S}}\log p(y|f(x, \Theta)),
\end{aligned}  \ee
where $\Mcal \subseteq \R^{m\times n}$ is an embedded Riemannian manifold, $\Theta \in \mathcal{M}$ is the parameter to be estimated, $\mathcal{S}$ is a collection of $|\mathcal{S}|$ data pairs $(x,y)$  with $x \in \mathcal{X}, y \in \mathcal{Y}$, $\mathcal{X}$ and $\mathcal{Y}$ are the input and output spaces, respectively, $f(\cdot, \Theta): \mathcal{X} \rightarrow \mathcal{Y}$ is a mapping from the input space to the output space, and  $p(y|f(x,\Theta))$ is the conditional probability of taking $y$ conditioning on $f(x,\Theta)$. 
If the conditional distribution is assumed to be Gaussian, the objective function in \eqref{prob} reduces to the square loss. 
%The logistic loss can be obtained if setting the conditional probability as the sigmoid-like functions. 
When the conditional distribution $p(y|f(x,\Theta))$ obeys the multinomial
distribution, the corresponding objective function is the cross-entropy loss.
As an aside, it is worth noting the equivalence between the negative log probability loss and
Kullback-Leibler (KL) divergence shown in \cite{martens2014new}. 

Let us take the low-rank matrix completion (LRMC) problem  \cite{boumal2015low,kasai2019riemannian} as an example and explain how it can be fitted into the form \eqref{prob}. The goal of LRMC is to recover a low-rank matrix from an observed matrix $X$ of size $n \times N$. Denote by $\Omega$ the set of indices of known entries in $X$, the rank-$p$ LRMC problem amounts to solving 
\be \label{prob:lrmc} \min_{U \in \mathrm{Gr}(n,p),A \in \R^{p\times N}} \frac{1}{2} \left\| \proj_{\Omega}(UA-X)\right\|^{2}, \ee 
where $\mathrm{Gr}(n,p)$ is the Grassmann manifold consists of all $p$-dimensional subspaces in $\R^n$. The operator $\proj_{\Omega}(X)$ is defined in an element-wise manner with $\mathcal{P}_{\Omega}( X_{ij} ) = X_{ij}$ if $(i,j)\in \Omega$ and $0$ otherwise.  Partitioning \( X=\left[x_1, \ldots, x_N \right] \) leads to the following equivalent formulation
\[ \min _{U \in \mathrm{Gr}(n,p), a_{i} \in \R^p} \frac{1}{2N}
\sum_{i=1}^N \left\|\proj_{\Omega_{x_i}}\left(U a_i- {x}_{i}\right) \right\|^{2}, \]
where $ x_i \in \mathbb{R}^{n}$ and the $j$-th element of $\proj_{\Omega_{x_i}}(v)$ is $v_j$ if $(i,j)\in \Omega$ and $0$ otherwise. Given $U$, we can obtain ${a}_{i}$ by solving a least squares problem, i.e.,
\[ a_i = a(U;x_i) :=  \argmin_{a} \| \proj_{\Omega_{x_i}}(U a - x_i) \|^2. \]
Then, the LRMC problem can be written as
\be \label{prob:lrmc-convert} \min_{U \in \mathrm{Gr}(n,p)} \;\; \Psi(U) := \frac{1}{2N}\sum_{i=1}^N \| \proj_{\Omega_{x_i}}\left(U a(U; x_i) - x_i\right) \|^2. \ee
For the Gaussian distribution $p(y|z) =\frac{1}{\sqrt{(2\pi)^n}}\exp(-\frac{1}{2}(y-z)^\top (y-z))$, it holds that $-\log p(y|z) = \frac{1}{2} \| y - z \|^2 + \frac{n\log(2\pi)}{2}$. Hence, problem \eqref{prob:lrmc-convert} is a special case of problem \eqref{prob}, in which $\mathcal{S} = \{(x_i, 0)\}_{i=1}^N$, $\mathcal{X} = \R^n$, $\mathcal{Y} = \R^n$, $f(x, U) = \proj_{\Omega_{x}}\left(U a(U; x) - x\right)$, $\Mcal = \mathrm{Gr}(n,p)$, and $p(y|z) =\frac{1}{\sqrt{(2\pi)^n}}\exp(-\frac{1}{2}(y-z)^\top (y-z))$. 
Other applications that can be fitted into the form \eqref{prob} will be introduced in Section \ref{sec:practicalRNGD}.

\subsection{Motivation of this work}
Since the calculation of the gradient of $\Psi$ in \eqref{prob} can be expensive when the dataset $\mathcal{S}$ is large, various approximate or stochastic methods for solving \eqref{prob} have been proposed. On the side of first-order methods, we have the stochastic gradient method \cite{robbins1951stochastic}, stochastic variance-reduced gradient method \cite{johnson2013accelerating}, and adaptive gradient methods \cite{duchi2011adaptive,kingma2014adam} for solving \eqref{prob} in the Euclidean setting (i.e., $\Mcal = \R^{m\times n}$). We refer the reader to the book \cite{lecun2015deep} for variants of these algorithms and a comparison of their performance. For the general manifold setting, by utilizing manifold optimization techniques \cite{opt-manifold-book,hu2020brief,boumal2020introduction}, 
Riemannian versions of the stochastic gradient method
\cite{bonnabel2013stochastic}, stochastic variance-reduced gradient method
\cite{sato2017riemannian,zhang2016riemannian,jiang2017vector}, and adaptive
gradient methods \cite{becigneul2018riemannian} have been developed. 

On the side of second-order methods, existing algorithms for solving \eqref{prob} in the Euclidean setting (i.e., $\mathcal{M} = \mathbb{R}^{m \times n}$) can be divided into two classes. The first is based on approximate Newton or quasi-Newton techniques; see, e.g., \cite{roosta2019sub,pilanci2017newton,byrd2016stochastic,yang2021stochastic,yang2021enhance,goldfarb2020practical,ren2021kronecker}. The second is
% i.e., \eqref{prob} without constraint. However, the efficiency of these methods for training large scale deep learning problems need to be further verified. 
the natural gradient-type methods, which are based on the Fisher information matrix (FIM) \cite{amari1996neural}. When the FIM can be approximated by a
Kronecker-product form, the natural gradient direction can be computed using
relative low computational cost. It is well known that second-order methods can accelerate convergence by utilizing curvature information. In particular, natural gradient-type methods can perform much better than the stochastic
gradient method
\cite{martens2015optimizing,yang2020sketchy,anil2020scalable,yang2021ng+,bahamou2022mini,nurbekyan2022efficient} in the Euclidean setting. The
connections between natural gradient methods and second-order methods have been
established in \cite{martens2014new}. Compared with the approximate
Newton/quasi-Newton-type methods, methods based on FIM are shown to be more efficient
when tackling large-scale learning problems. %The above mentioned second-order-type methods including natural gradient descent methods are all designed for the unconstrained setting, i.e., problem \eqref{prob} with $\Mcal = \R^{m\times n}$. 
For the general manifold setting, Riemannian stochastic quasi-Newton-type and
Newton-type methods
\cite{kasai2018riemannian,kasai2018inexact,zhang2020riemannian} have been proposed by
utilizing the second-order manifold geometry and variance reduction techniques. However, to the best of our knowledge, there is currently no Riemannian natural gradient-type method for solving \eqref{prob}. In view of the efficiency of Euclidean natural gradient-type methods, we are motivated to develop their Riemannian analogs for solving \eqref{prob}.

\subsection{Our contributions}
In this paper, we develop a new Riemannian natural gradient method for solving \eqref{prob}.
 Our main contributions are summarized as follows.
\begin{itemize}
	\item We introduce the Riemannian FIM (RFIM) and Riemannian empirical
        FIM (REFIM) to approximate the Riemannian Hessian. These notions extend the corresponding ones for the Euclidean setting \cite{amari1996neural,martens2014new} to the manifold setting. Then, we propose an adaptive regularized Riemannian natural gradient descent (RNGD) method. % under the stochastic trust region framework.
  		We show that for some representative applications, Kronecker-factorized approximations of RFIM and REFIM can be constructed, which reduce the computational cost of the Riemannian
        natural gradient direction. Our experiment results demonstrate that although RNGD is a second-order-type method, it has low per-iteration cost and enjoy favorable numerical performances.
  %RNGD is able to deal with both unbiased and biased estimates by controlling the accuracy of the estimates.   
	\item Under some mild conditions, we prove that RNGD globally converges to a stationary point of \eqref{prob} almost surely. Moreover, if the loss function satisfies certain convexity and smoothness conditions and the input-output map $f$ satisfies a Riemannian Jacobian stability condition, then we can establish the local linear---or, under the Lipschitz continuity of the Riemannian Jacobian of $f$, even quadratic---rate of convergence of the method by utilizing the notion of second-order retraction. We then show that for a two-layer neural network with batch normalization, the Riemannian Jacobian stability condition will be satisfied with high probability when the width of the network is sufficiently large. 
%	\item Numerical experiments for low rank matrix completion problem, subspace learning problem and deep learning problems are presented to show the powerness of our proposed RNGD method compared with state-of-the-art methods. 
\end{itemize}
\subsection{Notation} \label{subsec:notation}
For an $m\times n$ matrix $\Theta$, we denote its Frobenius norm by $\|\Theta\|$ and its vectorization by $\theta = \vect(\Theta) \in \mathbb{R}^{mn}$. For a function $h: \R^{m\times n}
\rightarrow \R$, we define its Euclidean gradient and Riemannian gradient on $\mathcal{M}$ by $\nabla h(\Theta) \in \mathbb{R}^{m \times n}$ and $\mbox{grad}\,h(\Theta) \in \mathbb{R}^{m \times n}$, respectively. For simplicity, we set $r=mn$. When no confusion can arise, we use
$\nabla h(\theta)$ and $\grad h(\theta)$ to denote the vectorizations of $\nabla
h(\Theta)$ and $\grad h(\Theta)$, respectively. We use $\nabla^2 h(\theta) \in
\R^{r\times r}$ and $\hess h(\theta) \in \R^{r\times r}$ to denote the Euclidean
Hessian and Riemannian Hessian of $h(\theta)$, respectively. We denote the tangent space to $\mathcal{M}$ at
$\Theta$ by $T_{\Theta}\Mcal$. We write $d \in T_{\theta} \Mcal$ to mean 
$\mat(d) \in T_{\Theta}\Mcal$, where $d \in \mathbb{R}^r$ and $\mat(d)$ converts $d$ into a
$m$-by-$n$ matrix. For a retraction $R$ defined on $\Mcal$, we write
$R_{\theta}(d) := \vect(R_{\Theta}(D)) \; \mathrm{for} \; D\in T_{\Theta}\Mcal$, $\theta = \vect(\Theta)$, and $d = \vect(D)$.  
We shall use $\theta$ and
$\Theta$ interchangeably when no confusion can arise. Basically, $\Theta$ is used when we want to utilize the manifold structure, while $\theta$ is used when we want to utilize the vector space structure of the ambient space.

\subsection{Organization} We begin with the preliminaries on manifold optimization and natural gradient methods in Section \ref{sec:preli}. In Section \ref{sec:RNGDM}, we introduce the RFIM and its
empirical version REFIM and derive some of their properties. Then, we present our proposed RNGD method by utilizing the RFIM
and REFIM. %The connections between RNGD and Riemannian Gauss-Newton method are investigated at the end of section \ref{sec:RNGDM}. 
In Section \ref{sec:practicalRNGD}, we discuss practical implementations of the RNGD method when problem \eqref{prob} enjoys certain Kronecker-product structure. In Section \ref{sec:convergence-analysis}, we
 study the convergence behavior of the RNGD method under various assumptions. Finally, we present numerical results in Section \ref{sec:numerics}.

\section{Preliminaries} \label{sec:preli}
\subsection{Manifold optimization}
Consider the optimization problem
\be \label{prob:gen} \min_{\Theta \in \Mcal} \;\; h(\Theta), \ee
where $\mathcal{M} \subseteq \mathbb{R}^{m \times n}$ is an embedded Riemannian manifold and $h:\R^{m\times n} \rightarrow \R$ is a smooth function. The design and analysis of numerical algorithms for tackling \eqref{prob:gen} have been extensively studied over the years; see, e.g., \cite{opt-manifold-book,hu2020brief,boumal2020introduction} and the references therein. One of the key constructs in the design of manifold optimization algorithms is the retraction operator. A smooth mapping $R:T\mathcal{M}:=\cup_{\Theta \in \Mcal} T_{\Theta} \Mcal \rightarrow \Mcal$ is called a retraction operator if  
\begin{itemize}
	\item $R_{\Theta}(0) = \Theta$,
	\item $\mathrm{D}R_{\Theta}(0)[\xi]:=\frac{\mathrm{d}}{\mathrm{d}t} R_{\Theta}(t\xi) \mid_{t=0} = \xi$, for all $\xi \in T_{\Theta}\Mcal$. 
\end{itemize}
We call $R$ a second-order retraction \cite[Proposition 5.5.5]{opt-manifold-book} if $\mathcal{P}_{T_{\Theta} \Mcal} \left( \frac{\mathrm{d}^2}{\mathrm{d}t^2} R_{\Theta}(t\xi) |_{t=0} \right) \\ = 0$ for all $\Theta \in \Mcal$ and $\xi \in T_{\Theta} \Mcal$. Some examples of second-order retraction can be found in \cite[Theorem 22]{absil2012projection}. In the $k$-th iteration, retraction-based methods for solving \eqref{prob:gen} update $\Theta^{k+1}$ by
\[ \Theta^{k+1} = R_{\Theta^k}(t d^k), \]
where $d^k$ is a descent direction in the tangent space $T_{\Theta^k} \Mcal$ and $t > 0$ is the step size. The retraction operator $R$ constrains the iterates on $\Mcal$. For a compact manifold, we have the following fact \cite{boumal2018global}, which will be used in our later analysis.
\begin{proposition}
	Let $\mathcal{M}$ be a compact embedded submanifold of $\mathbb{R}^{m \times n}$. For all $\Theta \in \mathcal{M}$ and $\xi \in T_{\Theta} \mathcal{M}$, there exists a constant $\alpha >0$ such that the following inequality holds:
	\be \label{eq:bound-retraction1} \|R_{\Theta}(\xi) - \Theta \| \leq \alpha \|\Theta \|, \; \forall \Theta
	\in \Mcal, \; \forall \xi \in T_{\Theta} \Mcal. \ee
\end{proposition}
\subsection{Natural gradient descent method}
The natural gradient descent (NGD) method was originally proposed in \cite{amari1996neural} to solve \eqref{prob} in the Euclidean setting (i.e., $\mathcal{M} = \mathbb{R}^{m \times n}$). Suppose that $y$ follows the conditional distribution
$P_{y|f(x,\Theta)}$. Consider the population loss under $P_{y|x}(\Theta) := P_{y|f(x,\Theta)}$, i.e., 
\be \label{eq:loss-expect} \Phi(\Theta) := -\E_{P_x} \left[ \E_{P_{y|x} (\Theta)}\log p(y|f(x, \Theta)) \right].\ee
When $P_{y|x}(\Theta)$
and $P_x$ are replaced by their empirical counterparts defined using $\mathcal{S}$, the population loss
$\Phi(\Theta)$  reduces to the empirical loss $\Psi(\Theta)$. Now, the FIM associated with $\Phi$ is defined as 
$$F(\theta) := \E_{P_x} [ \E_{P_{y|x}(\theta)} [  \nabla \log p(y|f(x,\theta)) \nabla \log p(y|f(x,\theta) )^\top ] ] \in \R^{r \times r}.$$
Under certain regularity condition \cite{flanders1973differentiation}, we can interchange the order of expectation and derivative to obtain $F(\theta) = \nabla^2 \Phi(\theta)$. In what follows, we assume that such a regularity condition holds. Since the distribution of $x$ is unknown, we set $P_x$ to be the empirical distribution defined by $\mathcal{S}$. In practice, we may only be able to get hold of an empirical counterpart of $P_{y|x}(\Theta)$. The empirical FIM (EFIM) associated with $\Psi$ is then defined by replacing $P_{y|x}(\Theta)$ with its empirical counterpart \cite{schraudolph2002fast}, i.e.,
\[ \bar{F}(\theta) := \frac{1}{|\mathcal{S}|} \sum_{(x,y) \in \mathcal{S}} \nabla \log p(y|f(x,\theta)) \nabla \log p(y|f(x,\theta))^\top.  \] 
With the FIM, the natural gradient direction is given by
\[ \tilde{\nabla} \Phi(\theta) := (F(\theta))^{-1} \nabla \Phi(\theta) \in \mathbb{R}^r. \]
It is shown in \cite[Theorem 1]{amari1998natural} and \cite[Proposition 1]{ollivier2017information} that $\tilde{\nabla} \Phi(\theta)$ is the steepest descent direction in the sense that 
\[ - \frac{\tilde{\nabla} \Phi(\theta)}{\|\nabla \Phi(\theta)\|_{(F(\theta))^{-1}}}=\lim _{\epsilon \rightarrow 0} \frac{1}{\epsilon} \argmin_{d \in \mathbb{R}^r: \operatorname{KL}\left(P_{x, y}(\theta+d) \| P_{x, y}(\theta)\right) \leq \epsilon^{2}/2} {\Phi(\theta+d)}, \]
where $\|\nabla \Phi(\theta)\|_{(F(\theta))^{-1}}:= \sqrt{\nabla \Phi(\theta)(F(\theta))^{-1} \nabla \Phi(\theta)}$.

In the $k$-th iteration, the iterative scheme of NGD for minimizing \eqref{eq:loss-expect} is  
\[ \theta^{k+1} = \theta^k - t_k \tilde{\nabla} \Phi(\theta^k), \]
where $t_k > 0$ is a proper step size. In the case where $F(\theta)$ is computationally expensive or inaccessible, we use the EFIM instead of the FIM. The connections between NGD and second-order methods are presented in \cite{martens2014new}.

\section{Riemannian natural gradient method} \label{sec:RNGDM}
\subsection{Fisher information matrix on manifold}
When the parameter to be estimated $\Theta$ lies on an embedded manifold $\Mcal$, the Euclidean natural gradient direction needs not lie on the tangent space to $\mathcal{M}$ at $\Theta$ and thus cannot be used as a search direction in retraction-based methods. 
To overcome this difficulty, we first introduce the RFIM, which is defined as
\be \label{eq:RFIM} F^R(\theta) :=  \E_{P_x} \left[ \E_{P_{y|x}(\theta)} \left[  \grad \log p(y|f(x,\theta))\grad \log p(y|f(x,\theta))^\top \right] \right] \in \R^{r\times r}, \ee
where $\grad \log p(y|f(x,\theta))$ is the Riemannian gradient of $\log p(y|f(x,\theta))$ with respect to $\theta$. Then, we define the Riemannian natural gradient direction $d^R(\theta)$ as 
\be \label{eq:RNGD} d^R(\theta) := (F^R(\theta))^{-1}\grad \Phi(\theta) \in \R^{r}. \ee
%Together with the simplicity of $F^R$ in calculations, we can also design efficient Riemannian natural gradient methods. 
 %A natural question is the correctness of the definition of RFIM. 
The following theorem justifies our definition of RFIM. It extends the corresponding results on FIM given in \cite[Theorem 1]{amari1998natural} and \cite[Proposition 1]{ollivier2017information}.  
\begin{Theorem} \label{thm:RFIM-geo}
	
	Let $\mathcal{M} \subseteq \mathbb{R}^{m \times n}$ be an embedded manifold and $\Phi: \mathcal{M} \rightarrow \mathbb{R}$ be the function given in \eqref{eq:loss-expect}. For any second-order retraction $R$ on $\mathcal{M}$, the steepest descent direction in the tangent space to $\Mcal$ at $\Theta$ is given by $- d^R(\theta)$ in \eqref{eq:RNGD}, i.e., 
	\[ \frac{-d^R(\theta)}{\| \grad \Phi(\theta) \|_{(F^R(\theta))^{-1}}} = \lim_{\epsilon \rightarrow 0} \frac{1}{\epsilon} \argmin_{d \in  T_{\theta} \Mcal: \E_{P_x}\left[\mathrm{KL}\left( P_{y|x}(R_{\theta}(d)) \| P_{y|x}(\theta) \right) \right] \leq \epsilon^2/2} \Phi(R_{\theta}(d)),   \]
	where $\| \grad \Phi(\theta) \|_{(F^R(\theta))^{-1}} = \sqrt{\grad \Phi(\theta)^\top (F^R(\theta))^{-1} \grad \Phi(\theta)}$.
\end{Theorem}
\begin{proof}
	For $\Theta \in \Mcal$, from the definition
	$$ \mathrm{KL}(P_{y|x}(\theta) || P_{y|x}(R_{\theta}(td))) = \E_{P_{y|x}(\theta)} \log p(y|f(x,\theta)) - \E_{P_{y|x}(\theta)} \log p(y| f(x, R_{\theta}(td))), $$
	we have 
	\[ \begin{aligned}
		\frac{\mathrm{d}}{\mathrm{d}t} \mathrm{KL}(P_{y|x}(\theta) || P_{y|x}(R_{\theta}(td))) \mid_{t=0} & = - \frac{\mathrm{d}}{\mathrm{d}t} \E_{P_{y|x}(\theta)} \log p(y|f(x, R_{\theta}(td))) \mid_{t=0} \\
		& =  - d^\top \nabla \E_{P_{y|x}(\theta)} \log p(y|f(x, \theta)).
	\end{aligned} \]
	By definition of the Riemannian gradient, we obtain 
	\[ d^\top \grad \mathrm{KL}(P_{y|x}(\theta) || P_{y|x}(R_{\theta}(td)))\mid_{t=0} =- d^\top \nabla \E_{P_{y|x}(\theta)} \log p(y| f(x, \theta)), \quad \forall d \in T_\theta \Mcal, \]
	where $\grad \mathrm{KL}(P_{y|x}(\theta) \parallel P_{y|x}(R_{\theta}(td)))\mid_{t=0} \in T_\theta \Mcal$. Then, we have
	\[ \grad \mathrm{KL}(P_{y|x}(\theta) \parallel P_{y|x}(R_{\theta}(td)))\mid_{t=0} = - \grad  \E_{P_{y|x}(\theta)} \log p(y|f(x, \theta)). \]
	Accordingly, using the Leibniz integral rule and the property of second-order retractions \cite[Proposition 5.5.5]{opt-manifold-book}, we have the second-order derivative
	\[ \begin{aligned}
		& \frac{\mathrm{d}^2}{\mathrm{d}t^2} \mathrm{KL}(P_{y|x}(\theta) || P_{y|x}(R_{\theta}(td))) \mid_{t=0} \\
		= & \E_{P_{y|x}(\theta)}[ d^\top \grad \log p(y|f(x, \theta)) \left(\grad \log p(y| f(x, \theta)) \right)^\top d ].
	\end{aligned}
	\]
	 It follows that $\grad  \E_{P_{y|x}(\theta)} \log p(y|f(x, \theta)) = 0$.
	%	 Define $F^R = \E_{P_{y|x}}\left(P_{T_\theta \Mcal} \left( \nabla \log p(x,y|\theta)\right) \right) \left(P_{T_\theta \Mcal} \left( \nabla \log p(x,y|\theta) \right) \right)^\top$, 
	By the definition of $F^R$, we conclude that
	$$ \E_{P_x}\mathrm{KL}(P_{y|x}(\theta) || P_{y|x}(R_{\theta}(d)) = \frac{1}{2}d^\top F^R(\theta) d + O(d^3), \quad \forall d \in T_{\theta} \Mcal.$$
	
	From the fact \cite[Proposition 1]{ollivier2017information} that
	\[  \frac{-A^{-1} \nabla h(\theta)}{\|\nabla h(\theta)\|_{A^{-1}}} = \lim_{\epsilon \rightarrow 0} \frac{1}{\epsilon} \argmin_{d :  \|d\|_A \leq \epsilon} h(\theta + d), \] 
	where $A$ is a positive definite matrix and $\|d\|_{A^{-1}} = \sqrt{d^\top A^{-1}d}$, we have 
	\be \label{eq:2nd-geo1}  \frac{-B^{-1} \nabla (\Phi \circ R_{\theta}) (0)}{\|\nabla (\Phi \circ R_{\theta}) (0)\|_{B^{-1}}} = \lim_{\epsilon \rightarrow 0} \frac{1}{\epsilon} \argmin_{d \; \in T_{\theta} \Mcal : \|d\|_A \leq \epsilon} \Phi(R_{\theta}(d)), \ee 
	where $B:T_{\theta} \Mcal \rightarrow T_{\theta} \Mcal $ is a positive definite linear operator.
	%	$R_{\theta} (d): T_{\theta}\Mcal \rightarrow \Mcal$ is a retraction operator and $R_{\theta}(td), t\in [0,1]$ introduce a curve on manifold $\Mcal$ with $R_{\theta}(0) = \theta$ and $\frac{d}{dt} R_{\theta}(td) \mid_{t=0} = d$. 
	Note that for all $u \in T_{\theta} \Mcal$, it holds that
	\[ \nabla (\Phi \circ R_{\theta}) (0) [u] = \nabla \Phi(R_{\theta}(0)) [\mathrm{D}R_\theta(0)[u]] = u^\top \grad \Phi(\theta). \]
	This gives
	\[  \nabla (\Phi \circ R_{\theta}) (0)  = \grad \Phi(\theta). \]
	Substituting the above into \eqref{eq:2nd-geo1} and letting $B= F^R(\theta)$, we have
	\be \label{eq:2nd-geo2}  \frac{-(F^R(\theta))^{-1}  \grad \Phi(\theta) }{\| \grad \Phi (\theta) \|_{(F^R(\theta))^{-1}}} = \lim_{\epsilon \rightarrow 0} \frac{1}{\epsilon} \argmin_{d \in T_{\theta} \Mcal : \|d\|_{F^R(\theta)} \leq \epsilon} \Phi(R_{\theta}(d)). \ee 
	Therefore, we conclude that
	$$
	\frac{-(F^R(\theta))^{-1} \grad \Phi(\theta) }{\| \grad \Phi(\theta) \|_{(F^R(\theta))^{-1}}}=\lim _{\epsilon \rightarrow 0} \frac{1}{\epsilon} \underset{d \in T_{\theta} \Mcal:  \E_{P_x} \left[ \operatorname{KL}\left(P_{y|x}(\theta) || P_{y|x}(R_{\theta}(d))\right) \right] \leq \epsilon^{2}/2}{\arg \min } \Phi(R_{\theta}(d))
	$$
	for any second-order retraction $R$.
\end{proof}

Note that the Riemannian Hessian \cite[Equation 7]{absil2013extrinsic} of $\Phi$ at $\theta$ along $u \in T_{\theta}\Mcal$ is given by
\[ \hess \Phi(\theta) [u] = \mathcal{P}_{T_\theta \Mcal} \left(\nabla^2 \Phi(\theta) [u] \right) - \mathcal{P}_{T_\theta\Mcal} \mathrm{D}_u (\grad \Phi(\theta)).  \] Hence, we have $\hess \Phi(\theta) = F^R(\theta)$ due to the fact that $ \grad \Phi(\theta) = 0$. Similar to EFIM,  we can define REFIM as
\be \label{eq:refim} \bar{F}^R(\theta) := \frac{1}{|\mathcal{S}|} \sum_{(x,y) \in \mathcal{S}}  \grad \log p(y|f(x,\theta)) \grad \log p(y|f(x,\theta))^\top. \ee

\subsection{Algorithmic framework}
In the $k$-th iteration, once we obtain an estimate $F_k$ of the RFIM associated with $\Phi$ or the REFIM associated with $\Psi$ at
$\theta^k$, the Riemannian natural gradient direction in the
tangent space to $\mathcal{M}$ at $\theta^k$ is computed by solving the following optimization problem:
\be \label{prob:NG-sub}
\begin{aligned}
	d^{k} = \argmin_{d \in T_{\theta^k}\Mcal} \quad m_k(d):= \Psi_k + \iprod{g^k}{d} + \frac{1}{2}\iprod{(F_k + \lambda_k I) d}{d}
\end{aligned}, \ee
where $\Psi_k$ and $g^k$ are stochastic estimates of $\Psi(\theta^k)$ and $\grad
\Psi(\theta^k)$, respectively and $\lambda_k > 0$ is usually updated adaptively by a
trust region-like strategy. Since $F_k + \lambda_k I : T_{\theta^k}\Mcal \rightarrow T_{\theta^k}\Mcal$ is positive definite and $g^k \in T_{\theta^k}\Mcal$, the solution of \eqref{prob:NG-sub} is 
$ d^k = -(F_k + \lambda_k I)^{-1} g^k.$
If the inverse of $F_k + \lambda_k I$  is costly to compute, then the truncated conjugated
gradient method can be utilized \cite{NocedalWright06}. 

Once $d^k$ is obtained, we construct a trial point 
\be \label{eq:trial} z^k = R_{\theta^k}(d^k). \ee
To measure whether $z^k$ leads to a sufficient decrease in the objective value, we first calculate the ratio $\rho_k$ between the reduction of $\Psi$ and the reduction of $m_k$. Since the exact evaluation of $\Psi$ is costly, one popular way \cite{chen2018stochastic} is to construct estimates $\Psi_k^0$ and $\Psi_k^{z^k}$ of $\Psi(\theta^k)$ and $\Psi(z^k)$, respectively. Then, we compute the ratio as
\be \label{eq:ratio} \rho_k = \frac{\Psi_k^{z^k} - \Psi_k^0}{m_k(d^k) - \Psi_k^0}.  \ee
Here, we take $\Psi_k = \Psi_k^0$ in the calculation of $m_k(d^k)$. Lastly, we perform the update
\be \label{eq:theta-update}
\theta^{k+1}=\left\{\begin{array}{ll}
	z^{k}, & \text { if } \rho_{k} \geq \eta_{1} \text{~and~} \|g^k\| \geq \frac{\eta_2}{\sigma_k}, \\
	\theta^{k}, & \text { otherwise, }
\end{array}\right.
\ee
where $\eta_1 \in (0,1)$ and $\eta_2 > 0$ are constants and $\sigma_k>0$ is used to control the regularization parameter $\lambda_k$. Indeed, to ensure the descent property of the original function
$\Psi$, some assumptions on the accuracy of the estimates of
$\Psi(\theta^k)$, $\Psi(z^k)$ and the model $m_k$ are needed, and they will be introduced  later in the convergence analysis. Due to the error in the estimates, 
the regularization parameter $\lambda_{k+1}$ should not only depend on the ratio $\rho_k$ but also on the norm of the estimated Riemannian gradient $g^k$. In particular, we set $\lambda_{k+1}:= \sigma_{k+1} \|g^{k+1}\|$ and update $\sigma_{k+1}$ as
\be \label{eq:mu-update}
\sigma_{k+1}  = \left\{\begin{array}{ll}
	\max\left\{ \sigma_{\min}, \frac{1}{\gamma} \sigma_{k} \right\}, & \text { if } \rho_k \geq \eta_1 \text{~and~} \|g^k\| > \frac{\eta_2}{\sigma_k}, \\
	\gamma \sigma_{k}, & \text { otherwise},\\
\end{array}\right.
\ee
where $\eta_1 \in (0,1)$, $\eta_2 > 0$ are as before and $\sigma_{\min} > 0$, $\gamma > 1$ are parameters. 
Our proposed RNGD method is summarized in Algorithm \ref{alg:RNGD}.
\begin{algorithm2e}[t] 
    \caption{Riemannian natural gradient descent (RNGD) for solving \eqref{prob}.}
	\label{alg:RNGD}
%	\lnlset{alg:TRNS0}{S0}
	Choose an initial point $\theta^0$ and parameters $\sigma_0 > 0$, $\sigma_{\min} > 0$, $\lambda_0 = \sigma_0 \|g^0\|$, $\eta_{1} \in (0,1)$, $\eta_2 > 0$, and $\gamma > 1 $. Set $k = 0$.
	
	\While{stopping conditions not met}
	{
		
%		\lnlset{alg:TRNS1}{S1}
		Compute the estimated Riemannian gradient $g^k$  and the estimated Riemannian Fisher information matrix $F_k$. \\
%		\lnlset{alg:TRNS2}{S2}
		Compute the negative natural gradient direction $d^k$ by solving \eqref{prob:NG-sub} and compute the trial point $z^k$ by \eqref{eq:trial}. \\
%		\lnlset{alg:TRNS3}{S3}
		Update $\theta^{k+1}$ based on \eqref{eq:theta-update}.\\
%		\lnlset{alg:TRNS4}{S4}
		Update $\lambda_{k+1}$ based on \eqref{eq:mu-update}.\\
		%negative curvature is detected. \\
		$k\gets k+1$.
	}
\end{algorithm2e}

\section{Practical Riemannian natural gradient descent methods} \label{sec:practicalRNGD}
From the definition of RFIM and REFIM in Section \ref{sec:RNGDM}, the computational cost of solving subproblem \eqref{prob:NG-sub} may be high because of the vectorization of $\Theta$. Fortunately, analogous to \cite{martens2015optimizing}, the Riemannian natural gradient direction can be computed with a relatively low cost if the gradient of a single sample is of low rank, i.e., for a pair of observations $(x,y) \in \mathcal{S}$ and $\psi(\Theta; x,y):= -\log p(y| f(x,\Theta))$, $\nabla\psi$ takes the form
\be \label{eq:kron-grad} \nabla \psi(\Theta; x,y) = G(x,y)A(x,y)^\top, \ee
where $G(x,y) \in \R^{m\times q}$ and $A(x,y) \in \R^{n \times q}$ with $q \ll \min(m,n)$. Let us now elaborate on this observation. 
%\subsection{RFIM with Kronecker-product form}/
%When the parameter $\Theta$ is constrained on a manifold $\Mcal$, the geometric operators on manifold should be investigated as introduced in Section \ref{sec:RNGDM}. 
%Assuming $W_l \in \R^{n_l \times n_{l-1}}$ lies on $\Mcal$, 
%the Riemannian gradient $\mathcal{D}^R W_l$ is then given by
%\[ \mathcal{D}^R W_l = \proj_{T_{W_l} \Mcal} \left(\mathcal{D} W_l \right) = \proj_{T_{W_l} \Mcal} (g_la_{l-1}^\top). \]

Recall that the Riemannian gradient of $\psi$ is given by
\[ \grad \psi(\Theta; x,y) = \proj_{T_{\Theta} \Mcal} (\nabla \psi(\Theta; x,y)). \]
When $\nabla\psi$ has the form \eqref{eq:kron-grad}, the linearity of the projection operator implies that
\be \label{eq:efim-kron} \begin{aligned}
	F^R(\theta) & = \E_{P_{x,y}(\theta)} \left[\grad \psi(\theta; x,y) \grad \psi(\theta; x,y)^\top\right] \\
	& \approx  \mathcal{P} \left(\E_{P_{x,y}(\theta)} \left[ A(x,y) A(x,y)^\top \right] \otimes \E_{P_{x,y}(\theta)} \left[ G(x,y) G(x,y)^\top \right]\right) \mathcal{P},
\end{aligned}  \ee
where $P_{x,y}(\theta)$ is the joint distribution of $(x,y)$ given $\theta$, $\mathcal{P} \in \R^{r \times r}$ is the matrix representation of $\proj_{T_{\Theta}\Mcal}$ (note that $\mathcal{P}^\top = \mathcal{P}$ due to the symmetry of orthogonal projection operators), and the approximation is due to the assumption that $A(x,y)$ and $G(x,y)$ are approximately independent; see also \cite[Theorem 1]{grosse2016kronecker} for a use of such an assumption to derive a simplified form of the FIM. By replacing $P_{x,y}(\theta)$ with its empirical distribution observed from $\mathcal{S}$,  an approximate REFIM is given by
\be \label{eq:refim-kron} \bar{F}^R(\theta) \approx \mathcal{P} \left(\left[\frac{1}{|\mathcal{S}|} \sum_{(x,y)\in \mathcal{S}}  A(x,y) A(x,y)^\top \right] \otimes  \left[ \frac{1}{|\mathcal{S}|} \sum_{(x,y)\in \mathcal{S}}  G(x,y) G(x,y)^\top \right] \right) \mathcal{P}. \ee
%Although the projection operator can often be calculated cheaply,
 When a direct inverse of
 $\bar{F}^R(\theta)$ is expensive to compute, the truncated
 conjugated gradient method can be used. % when the direct inverse of $\bar{F}^R(\theta)$ is not as easy as FIM. % one may still be concerned with the calculation of the natural gradient direction since the direct inverse of $\bar{F}^R(\theta)$ is not as easy as FIM. Otherwise, the truncated conjugated gradient direction can be used to obtain an inexact natural gradient direction. 
 In preparation for the applications, we now show how to construct computationally efficient approximations of the RFIM and REFIM on the Grassmann manifold.
\subsection{RFIM and REFIM on Grassmann manifold}
If the matrix representation $\mathcal{P}$ of the projection operator $\proj_{T_{\Theta}\Mcal}$ has dimensions $m$-by-$m$ or $n$-by-$n$, i.e.,
\[ \grad \psi(\Theta; x,y) = B_1 G(x,y) A(x,y)^\top \;\; \mathrm{or} \;\; \grad \psi(\Theta; x,y) = G(x,y) A(x,y)^\top B_2 \]
with $B_1 \in \R^{m \times m}$ and $B_2 \in \R^{n \times n}$, then we can approximate the RFIM in \eqref{eq:efim-kron} by
\[ F^R(\theta) \approx \E_{P_{x,y}(\theta)} \left[  A(x,y) A(x,y)^\top \right] \otimes \E_{P_{x,y}(\theta)} \left[ B_1G(x,y) G(x,y)^\top B_1 \right] \]
or 
\[ F^R(\theta) \approx \E_{P_{x,y}(\theta)} \left[ B_2A(x,y) A(x,y)^\top B_2 \right] \otimes \E_{P_{x,y}(\theta)} \left[ G(x,y) G(x,y)^\top \right]. \]
Moreover, if we replace $P_{x,y}(\theta)$ by its empirical distribution observed from $\mathcal{S}$, then we can approximate the REFIM in \eqref{eq:refim-kron} by
\[
	\bar{F}^R(\theta) \approx \left(\frac{1}{|\mathcal{S}|} \sum_{(x,y)\in \mathcal{S}} A(x,y) A(x,y)^\top \right) \otimes \left(\frac{1}{|\mathcal{S}|} \sum_{(x,y)\in \mathcal{S}} B_1G(x,y) G(x,y)^\top B_1 \right) \] 
	or
	\[ 
	\bar{F}^R(\theta)  \approx \left(\frac{1}{|\mathcal{S}|} \sum_{(x,y)\in \mathcal{S}} B_2A(x,y) A(x,y)^\top B_2\right) \otimes \left(\frac{1}{|\mathcal{S}|} \sum_{(x,y)\in \mathcal{S}} G(x,y) G(x,y)^\top \right).
\]
Note that the Kronecker product form allows the inverse of $\bar{F}^R(\theta)$ to be calculated efficiently by inverting
two smaller  matrices \cite{martens2015optimizing}.
%the direct inverse of $\bar{F}^R(W_l)$ can be written as
%\[ \left[\bar{F}^R(W_l)\right]^{-1} = [B_1A_{l-1}B_1]^{-1} \otimes G_l^{-1} \;\; \mathrm{or} \;\; \left[\bar{F}^R(W_l)\right]^{-1} = A_{l-1}^{-1} \otimes [B_2G_lB_2]^{-1}.  \]
A typical manifold that yields the above Kronecker product representations is the Grassmann manifold $\mathrm{Gr}(m,n)$, which consists of all $n\;(\mathrm{resp.}, \, m)$ dimensional subspaces in $\R^m \; (\mathrm{resp.}, \, \R^n)$ if $m \geq n$ (resp., $m < n$). The matrix representation of the projection operator is $B_1 = I_{m} -\Theta \Theta^\top$ ($m\geq n$) or $B_2 = I_{n} - \Theta^\top \Theta$ ($m < n$). In what follows, we derive the RFIMs associated with three concrete applications involving the Grassmann manifold and explain how they can be computed efficiently.

\subsection{Applications}
\subsubsection{Low-rank matrix completion}

For simplicity, we derive the RFIM associated with problem \eqref{prob:lrmc-convert} for the fully observed case, i.e., $\Omega = \{1,\ldots, n\}\times \{1,\ldots, N\}$. One can derive the RFIM for the partly observed case in a similar fashion. By definition, we have $f(x, U) = U a(U; x) - x$ and $\psi(U; x,y) = -\log p(y|f(x, U)) =\frac{1}{2}\| f(x, U) - y \|^2 + \frac{n\log(2\pi)}{2}$. It follows from \cite{boumal2015low} that the Jacobian of $a$ along a tangent vector $H \in T_U\mathrm{Gr}(n,p)$ is given by $J_a(U;x)[H] = H^\top x$ and its adjoint $J^\top_a(U;x)$ satisfies $J^\top_a(U;x)[v] = x^\top v$ for $v \in \R^p$.  
%let the loss $L(f(x_i, U), y_i) := \|f( x_i, U)\|^2$. Then $\Psi(U) = \frac{1}{N} \sum_{i=1}^N \left\{ \psi(U;x_i, y_i):=L(f(x_i, U), y_i) \right\}$. 
The Riemannian gradient of $\psi(\cdot;x,y)$ is 
\[ \begin{aligned}
	\grad \psi(U; x,y) = & (I-UU^\top)((Ua(U;x)-x - y)a(U; x)^\top) \\
	& + (I-UU^\top) x(Ua(U;x) - x- y)^\top U.
\end{aligned} \]
By assuming that the residual $Ua(U;x) - x$ is close to zero, we have$(I - UU^\top)x \approx (I - UU^\top)Ua(U;x) = 0$. This leads to the following approximate Riemannian gradient of $\psi(\cdot; x, y)$: 
\be \label{eq:approx-grad-lrmc} \grad \psi(U; x,y) \approx (I-UU^\top)((Ua(U;x)-x - y)a(U; x)^\top). \ee
Plugging the above approximation into \eqref{eq:efim-kron} leads to 
\[ \begin{aligned}
	F^R(u) = & \E_{P_x} \left[ \E_{P_{y|x}(u)}\left[ \grad \psi(u; x, y) \grad \psi(u;x, y)^\top \right] \right] \\
	 \approx & \E_{P_x} \left[ \E_{P_{y|x}(U)} \left[ [a(U;x)a(U;x)^\top] \otimes \left[(I-UU^\top) (Ua(U;x) - x - y) \right. \right. \right. \\
	 & \left. \left. \left. (Ua(U;x) - x - y)^\top (I-UU^\top) \right]  \right] \right] \\
 \approx & \left[\frac{1}{N} \sum_{i=1}^N a(U;x_i)a(U;x_i)^\top \right] \otimes (I - UU^\top), 
\end{aligned} \]
where $u = \vect(U)$ is the vectorization of $U$, the second line is due to \eqref{eq:approx-grad-lrmc}, $\vect(u v^\top) = v\otimes u$, $(A\otimes B)^\top = A^\top \otimes B^\top$, and $(A\otimes B)(A^\top \otimes B^\top) = (AA^\top) \otimes (BB^\top)$, and the last line follows from $\E_{P_{y|x}(U)}\left[(Ua(U;x) - x - y) (Ua(U;x) - x - y)^\top \right] = I$ and by substituting $P_x$ with its empirical distribution. For $H \in T_U\mathrm{Gr}(n,p)$, we have 
\be \label{eq:rfim-lrmc} \begin{aligned}
	\mathrm{mat}(F^R(u)[\vect(H)]) & \approx \left[\frac{1}{N} \sum_{i=1}^N a(U;x_i)a(U;x_i)^\top \right] \otimes (I - UU^\top) \vect(H) \\
	& = H \left[\frac{1}{N} \sum_{i=1}^N a(U; x_i)a(U; x_i)^\top \right],
\end{aligned} \ee
where $\mathrm{mat}(b)$ converts the vector $b \in \R^{np}$ into an $n$-by-$p$ matrix and the equality follows from $(I - UU^\top)H = H$.
%\[ F^R(\vect(U)) = \frac{1}{2N} \sum_{i=1}^N \left[J^R(x_i, U)\right]^\top J^R(x_i, U), \]
%where $J^R(x_i, U)$ is the Riemannian Jacobian of $f(x_i, U)$ at $U$. Let us show the calculations of $J^R(x_i, U)$. It follows from \cite{boumal2015low} that $\mathrm{D}(a_i(U))[H] = H^\top x_i$. Then the action of $J^R(U)$ along a tangent vector $H$ is 
%\[ J^R(x_i, U)[H] = \proj_{T_U\mathrm{Gr}(n,p)} \left[Ha_i(U) + U \mathrm{D}a_i(U)[H] \right] = Ha_i(U),  \]
%where $\proj_{T_U\mathrm{Gr}(n,p)} = I - UU^\top$ and the second equality is due to the expression of $\mathrm{D}a_i(U)[H]$ and $U^\top H = 0$. Following the definition of adjoint operator, it holds that $J^R(x_i, U)^\top[H] = Ha_i(U)^\top$. Combining the expressions of $J^R(x_i, U)[H]$ and $J^R(x_i, U)^\top[H]$ yields
%\be \label{eq:rfim-lrmc} F^R(U)[H] = \frac{1}{2N} H \sum_{i=1}^N \left(a_i(U)a_i(U)^\top \right). \ee
For the partly observed case, %the expression of RFIM is more complicated. Motivated by \cite{boumal2015low}, which utilizes the fully observed Hessian to precondition the exact Hessian, 
the matrix $F^R(u)$ defined in the above equation can serve as a good approximation of
the exact RFIM. Note that $\frac{1}{N} \sum_{i=1}^N a(U; x_i)a(U; x_i)^\top
\in \R^{p\times p}$ is of low dimension since the rank $p$ is usually
small. Thus, the Riemannian natural gradient direction can be calculated with a relatively low cost. 
\subsubsection{Low-dimension subspace learning} 
In multi-task learning \cite{ando2005framework,mishra2019riemannian},
different tasks are assumed to share the same latent low-dimensional feature
representation. Specifically, suppose that the $i$-th task  has the 
training set $X_i\in \R^{d_i\times n}$ and the corresponding label set $y_i\in
\R^{d_i}$ for $i=1,\ldots,N$. The multi-task feature learning problem can then be formulated as
\begin{equation}
	\label{prob:subspace-learning}
	\min_{U \in \text{Gr}(n,p)} \Psi(U) = \frac{1}{2N}\sum_{i=1}^N \|X_i U w(U; X_i, y_i) - y_i\|^2,
\end{equation}
where $w(U; X_i, y_i) = \arg\min_{w}  \frac{1}{2} \|X_i U w - y_i\|^2 + \lambda \|w\|^2$ and $\lambda > 0$ is a regularization parameter. Suppose that $d_1 = \cdots = d_N = d$. Then, problem \eqref{prob:subspace-learning} has the form \eqref{prob}, where $\mathcal{S} = \{ ((X_i, y_i), 0) \}_{i=1}^N$, $\mathcal{X} = \R^{d \times (n+1)}, \; \mathcal{Y} = \R^d$,  $f(X, y, U) = X U w(U; X,y) - y $, and $p(z|f(X, y, U)) =\frac{1}{\sqrt{(2\pi)^{d}}}\exp(-\frac{1}{2}(z-f(X, y, U))^\top (z-f(X, y, U)))$. By ignoring the constant $\frac{d\log(2\pi)}{2}$ and slightly abusing the notation, we define $\psi(U; X, y, z) =  \frac{1}{2}\|X U w(U; X, y) - y - z\|^2$. Using the optimality of $w(U; X,y)$, we have $U^\top X^\top (X U \\ w(U; X,y) - y) + \lambda w(U; X, y) =0$. 
%Hence, $U^\top X^\top (X U w(U; x, y) - y - z) $ is close to $0$ as $\lambda\|w(U; x, y)\|$ and $\|U^\top X^\top z\|$ go to zero. 
Then, we can compute the Euclidean gradient of $\psi(\cdot;X,y,z)$ as
\[ \begin{aligned}
	& \nabla \psi(U; X, y, z) \\
	 = & X^\top(X U w(U; X, y) - y -z) w(U; X,y)^\top + J_w^\top(U) \left[ U^\top X^\top(X U w(U; X, y) - y - z) \right] \\
	\approx & X^\top(X U w(U;X,y) - y) w(U; X,y)^\top,
\end{aligned}
 \]
 where $J_w(U)$ is the Jacobian of $w(U;X,y)$, $J_w^\top(U)$ denotes the adjoint of $J_w(U)$, and the approximation holds for small $\lambda$ and $\|z\|$. Note that $z$ will lie in a small neighborhood of zero with high probability if $f(X,y, U)$ is close to 0. Besides, $z$ is always zero in the dataset $\mathcal{S}$.
 With the above, an approximate Riemannian gradient of $\psi(\cdot; X,y,z)$ is given by 
 \be \label{eq:app-grad-ldsl} \grad \psi(U; X, y, z) \approx (I-UU^\top) X^\top(X U w(U; X,y) - y - z) w(U; X,y)^\top. \ee 
Consequently, we have
\begin{equation} \label{eq:fim-subspace-learning}
	\begin{aligned} 
		F^R(u) & = \E_{P_{(X,y)}} \left[ \E_{P_{z|(X,y)}(u)}[ \grad \psi(u; X, y, z) \grad \psi(u; X, y, z)^\top] \right]\\
		\approx & \frac{1}{N}\sum_{i=1}^N ( w_i \otimes ((I-UU^\top) X_i^\top))( w_i \otimes ((I-UU^\top) X_i^\top))^\top      \\
		=& \frac{1}{N}\sum_{i=1}^N \left[ ( w_i w_i^\top) \otimes ((I-UU^\top)X_i^\top X_i (I-UU^\top))\right] \\
		\approx & \frac{1}{N}\left[  \sum_{i=1}^N  w_i w_i^\top \right] \otimes \left[ \frac{1}{N} \sum_{i=1}^N (I-UU^\top)X_i^\top X_i (I-UU^\top)\right],
	\end{aligned}
\end{equation}
where $u = \vect(U)$ is the vectorization of $U$, $w_i:= w(U;X_i,y_i)$, the second line follows from \eqref{eq:app-grad-ldsl}, $\E_{P_{z|(X,y)}(u)}[ (X U w(U; X,y) - y - z) (X U w(U; X,y) - y - z)^\top ] = I$, and the empirical approximation of $P_{(X,y)}$, and the last line holds under the same condition as in \eqref{eq:efim-kron}. Though the construction of $F^R(u)$ is for the case $d_1 = \cdots = d_N$, it can be easily extended to the case where the $d_i$'s are not equal. 
% By a similar idea used in \cite{grosse2016kronecker}, we further approximate $F^R$ by $ \left[ \sum_{i=1}^N  ( w_i w_i^\top) \right] \otimes \left[ \sum_{i=1}^N ((I-UU^\top)X_i^\top X_i (I-UU^\top)\right]$, which gives the natural gradient direction by calculating inverses of a $p$-by-$p$ matrix and an $n$-by-$n$ matrix. 
\subsubsection{Fully connected network with batch normalization}\label{sec:3.2}
Consider an $L$-layer neural network with input $a_{0} = x$. In the $l$-th layer, we have
\be \label{eq:fcn}
s_{l}=W_{l} a_{l-1} + b_l, \; t_{l,i} = \frac{s_{l,i}-\mathbb{E}(s_{l,i})}{\operatorname{Var}(s_{l,i})} \times \gamma_{l,i} + \beta_{l,i},\; i=1,\ldots,n_l, \; a_{l}=\varphi_{l}\left(t_l\right), 
\ee
%and
%\be \label{eq:bn}	\operatorname{BN}(s_{l,i}) =  \frac{s_{l,i}-\mathbf{E}(s_{l,i})}{\operatorname{Var}(s_{l,i})} \times \gamma + \beta 
%\ee
where $\varphi_{l}$ is an element-wise activation function, $W_{l} \in \R^{n_{l}\times n_{l-1}}$ is the weight, $b_l \in \R^{n_l}$ is the bias, $s_{l,i}$ is the $i$-th component of $s_l \in \R^{n_l}$, $\gamma_{l,i}, \beta_{l,i} \in \R$ are two learnable parameters, $\mathrm{Var}(s_{l,i})$ is the variance of $s_{l,i}$, and $f(x, \Theta)=a_{L} \in \R^{m}$ is the output of the network with $\Theta$ being the collection of parameters $\{ W_l, b_l, \gamma_l, \beta_l \}$. By default, the elements of $\gamma_{l,i}$ are set to 1 and the elements of $\beta_{l,i}$ are set to 0. In \cite{ioffe2015batch}, $t_{l,i}$ is called the batch normalization of $s_{l,i}$.

Given a dataset $\mathcal{S}$, our goal is to minimize the discrepancy between the network output $f(x, \Theta)$ and the observed output $y$, namely,
\be \min_{\Theta} \;\; \Psi(\Theta) = -\frac{1}{|\mathcal{S}|} \sum_{(x,y) \in \mathcal{S}}\log p(y|f(x, \Theta)). \ee
By \cite{cho2017riemannian}, each row of $W_l$ lies on the Grassmann manifold $\mathrm{Gr}(1,n_{l-1})$. It follows that $W_l$ lies on the product of Grassmann manifolds, i.e., 
$ W_l \in \mathrm{Gr}(1,n_{l-1}) \times \cdots \times \mathrm{Gr}(1,n_{l-1}) \in \R^{n_l \times n_{l-1}}$.
The remaining parameters lie in the Euclidean space. Rather than batch
normalization, layer normalization \cite{ba2016layer}
and weight normalization \cite{salimans2016weight} have also been widely investigated in the study of deep neural networks, where $\vect(W_l) \in \mathrm{Gr}(n_l \times n_{l-1}, 1)$ and $W_l \in \mathrm{Sp}(n_{l-1} -1 ) \times \cdots \times \mathrm{Sp}(n_{l-1}-1) \in \R^{n_l \times n_{l-1}}$ with $\mathrm{Sp}(n_{l-1} -1):=\{u \in \R^{n_{l-1}} : \|u\| = 1 \}$, respectively.

By back-propagation, the Euclidean gradient of $\Psi$ with respect to $W_l$ is given by 
$$
g_{l} \leftarrow \mathrm{D} a_{l} \odot \varphi_{l}^{\prime}\left(t_l\right) \odot \mathrm{D} t_l,  \quad \nabla \Psi(W_{l}) \leftarrow g_{l} a_{l-1}^{\top}, \quad \mathrm{D} a_{l-1} \leftarrow W_{l}^{\top} g_{l}.
$$
In particular, we see that $\nabla \Psi(W_{l})$ has the Kronecker product form \eqref{eq:kron-grad}.
Moreover, note that $\Psi(w_{l,i}) = \Psi(cw_{l,i}), \;  \forall c\ne 0$. Now, we compute
\[ \nabla \Psi(w_{l,i})w_{l,i}^\top = \lim_{t \rightarrow 0} \frac{\Psi(w_{l,i} + t w_{l,i}) - \Psi(w_{l,i})}{t} = 0. \] 
By definition of the projection operator defined on the product of Grassmann manifolds, the Riemannian gradient $\mbox{grad}\,\Psi(W_l)$ is actually the same as the Euclidean gradient $\nabla \Psi(W_l)$. Specifically, for the $i$-th row of $\grad \Psi(W_l)$, we have
\[ \left[\grad \Psi(W_l) \right]_i = \grad \Psi(w_{l,i}) = \nabla \Psi(w_{l,i})
- \nabla \Psi(w_{l,i}) w_{l,i}^\top w_{l,i}  = \nabla \Psi(w_{l,i}). \]
Therefore, the RFIM coincides
with the FIM. The inverse of $F^R(\theta)$ can be computed easily when the FIM has a Kronecker product form.

\section{Convergence Analysis} \label{sec:convergence-analysis}
In this section, we study the convergence behavior of the RNGD method (Algorithm \ref{alg:RNGD}). 
\subsection{Global convergence to a stationary point}
To begin, let us extend some of the definitions used in the study of Euclidean stochastic trust-region methods (see, e.g., \cite{chen2018stochastic}) to the manifold setting. 
\begin{definition} \label{def:full-linear}
	Let $\kappa_{\rm ef}, \kappa_{\rm eg} > 0$ be given constants. A function $m_k$ is called a $(\kappa_{\rm ef}, \kappa_{\rm eg})$-fully linear model of $\Psi$ on $B_{\theta^k}(0, 1/\sigma_k)$ if for any $y \in B_{\theta^k}(0, 1/\sigma_k)$,
	\be \label{eq:full-linear} 
	\| \nabla (\Psi \circ R_{\theta^k})(y) - \nabla m_k(y) \| \leq \frac{\kappa_{\rm eg}}{\sigma_k} \quad \textrm{and}\quad |\Psi \circ R_{\theta^k}(y) - m_k(y)| \leq \frac{\kappa_{\rm ef}}{\sigma_k^2},  \ee
	where $B_{\theta}(0, \rho) := \left\{d \in T_{\theta}\Mcal : \|d\| \leq \rho \right\}$. 
\end{definition}

\begin{definition} \label{def:acc-est}
	Let $\epsilon_F, \sigma_k>0$ be given constants. The quantities $\Psi_{k}^{0}$ and $\Psi_{k}^{z^k}$ are called $\epsilon_{F}$-accurate estimates of $\Psi\left(\theta^{k}\right)$ and $\Psi_k\left(z^k\right)$, respectively if
	\be \label{eq:acc-est}
	\left|\Psi_{k}^{0}-\Psi\left(\theta^{k}\right)\right| \leq \frac{\epsilon_{F}}{\sigma_{k}^{2}} \quad \text { and } \quad\left|\Psi_{k}^{z^k}-\Psi_k\left(z^k\right)\right| \leq \frac{\epsilon_{F}}{\sigma_{k}^{2}},
	\ee
	where $z^k$ is defined in \eqref{eq:trial}.
\end{definition}

Analogous to \cite{chen2018stochastic,wang2019stochastic}, the inequalities \eqref{eq:full-linear} and \eqref{eq:acc-est} can be guaranteed when $\mathcal{M}$ is compact, the number of samples is large enough, and $\nabla (\Psi\circ R)$ is Lipschitz continuous. 

Next, we introduce the assumptions needed for our convergence analysis. Their Euclidean counterparts can be found in, e.g., \cite[Assumptions 4.1 and 4.3]{chen2018stochastic}.

\begin{assumption} \label{assum:f}
	
	Let $\theta^0 \in \mathbb{R}^r, \sigma_{\min} > 0$ be given. Let $\mathcal{L}(\theta^0)$ denote the set of iterates generated by Algorithm \ref{alg:RNGD}. Then, the function $\Psi$ is bounded from below on $\mathcal{L}(\theta^0)$. Moreover, the function $\Psi \circ R$ and its gradient $\nabla (\Psi\circ R)$ are $L$-Lipschitz continuous on the set
	\[ \mathcal{L}_{\rm enl}(\theta^0) = \bigcup_{\theta \in \mathcal{L}(\theta^0)} B_{\theta}\left(0,\frac{1}{\sigma_{\min}}\right). \]
	%	\bigcup_{\theta \in \mathcal{L}(\theta^0)} \left\{R_\theta(y): y \in B_{\theta}(0,\frac{1}{\sigma_{\min}})\right\}. \]
\end{assumption}
%To estimate the decrease, we also need the upper boundedness of $F^R(\theta^k)$.
\begin{assumption} \label{assum:fisher}
	The RFIM or REFIM $F_k$ satisfies $\|F_k\|_{\rm op} \leq \kappa_{\rm fim}$ for all $k \ge 0$, where $\| \cdot \|_{\rm op}$ is the operator norm.
%	the matrix $2$-norm of $F_k$ is bounded from above by a constant $\kappa_{fim} > 0$, i.e., $\| F_k \|_2 \leq \kappa_{fim}$ for all $k$.
%	\be \label{eq:fisher-bound} \| F_k \| \leq \kappa_{fim}. \ee
\end{assumption}

With the above assumptions, we can prove the convergence of Algorithm \ref{alg:RNGD} by adapting the arguments in \cite{chen2018stochastic}. The main difference is that our analysis makes use of the pull-back function $\Psi\circ R$ and its Euclidean gradient; see  Definitions \ref{def:full-linear} and $\ref{def:acc-est}$. 
\begin{theorem} \label{thm:lim-con}
	Suppose that Assumptions \ref{assum:f} and \ref{assum:fisher} hold, $m_k$ is
    a $(\kappa_{\rm ef}, \kappa_{\rm eg})$-fully linear model for some $\kappa_{\rm ef}, \kappa_{\rm eg} > 0$, and the estimates $\Psi_k^0$ and
    $\Psi_k^{z^k}$ are $\epsilon_{F}$-accurate for some $\epsilon_F>0$. Furthermore, suppose that  $\eta_2 \geq \max\left\{\kappa_{\rm fim}, \frac{16 \kappa_{\rm ef}}{1-\eta_1}\right\}$ and $ \epsilon_{F} \leq \min \left\{ \kappa_{\rm ef}, \frac{1}{32} \eta_1 \eta_2 \right\}$.  
	Then, the sequence of iterates $\{ \theta^k \}$ generated by Algorithm \ref{alg:RNGD} will almost surely satisfy
	$$
	\liminf_{k \rightarrow \infty}\left\|\grad \Psi(\theta^{k})\right\|=0.
	$$
\end{theorem}
\begin{proof}
	One can prove the conclusion by following the arguments in \cite[Theorem
    4.16]{chen2018stochastic}. We here present a sketch of the proof. 
%    In the $k$-th iteration, 
%    we define $\vartheta^k, \Psi_k$, $\hat{\Psi}_k$ and $\Sigma^k$ as the random vectors with realizations $\theta^k, \Psi_k^0, \Psi_k^{z^k}$ and $\sigma_k$, respectively. Let $\mathcal{F}_{k-1}^{M \cdot \Psi}$ be the $\sigma$-algebra generated by the random models $M_0, \ldots, M_{k-1}$ and the random function values $\Psi_0, \ldots, \Psi_{k-1}$, $\hat{\Psi}_0, \ldots, \hat{\Psi}_{k-1}$, and $m_k = M_k(\omega)$ be the realization.
    Define $\mathcal{F}_k$ as the $\sigma$-algebra generated by $\Psi_1^0,\Psi_1^{z^1}, \ldots, \Psi_k^0,\Psi_k^{z^k}$ and $m_1, \ldots, m_k$.
    Consider the random function $\Phi_{k}=v \Psi(\theta^k) +(1- v)
    /\sigma_{k}^{2}$, where $v \in(0,1)$ is fixed.  
    The idea is to prove that there exists a constant $\tau>0$ such that for all $k$,
	\be \label{eq:desc-exp}
	\E\left[\Phi_{k+1}-\Phi_{k} \mid \mathcal{F}_{k-1}\right] \leq -\frac{\tau}{\sigma_k^2} <0.
	\ee 
	Summing \eqref{eq:desc-exp} over $k \ge 1$ and taking expectations on both sides lead to $\sum_{k=1}^\infty 1/\sigma_k^2 < \infty$. 
	The inequality \eqref{eq:desc-exp} can be proved in the following steps. Firstly, a decrease on
    $\Psi$ of order $-\mathcal{O}(1/\sigma_k^2)$ can be proved using the
    fully linear model approximation and the positive definiteness of $F_k +
    \sigma_k \|g^k\| I$ with a sufficiently large $\sigma_k$. Secondly, the trial point $z^k$ will
    be accepted provided that the estimates $\Psi_k^0$ and $\Psi_k^{z^k}$ are
    $\epsilon_{F}$-accurate with sufficiently small $\epsilon_{F}$ and large $\sigma_k$. In addition, with $\eta_2 \ge \max\left\{ \kappa_{\rm fim}, \frac{16 \kappa_{\rm ef}}{1-\eta_1} \right\}$, if $z^k$ is accepted (i.e., $\theta^{k+1} = z^k$), then a decrease of
    $-\mathcal{O}(1/\sigma_k^2)$ on $\Psi$ can always be  guaranteed when
    $\epsilon_{F} \leq \min \left\{ \kappa_{\rm ef}, \frac{1}{32} \eta_1 \eta_2
    \right\}$ based on the update scheme \eqref{eq:mu-update}. On the other hand, if $z^k$ is rejected (i.e., $\theta^{k+1} = \theta^k$), then $\E\left[\Phi_{k+1} - \Phi_k | \mathcal{F}_{k-1}\right] = (1-v) (1/\gamma^2 -1)/\sigma_k^2$. By choosing $v$ to be sufficiently close to 1, the inequality \eqref{eq:desc-exp} holds for any $k$. 
    
    Now, we have $\sigma_k \rightarrow \infty$ as $k \rightarrow \infty$ with probability 1. If there exist $\epsilon > 0$ and $k_0 \ge 1$ such that
    $\|\grad \Psi(\theta^k)\| \geq \epsilon \; \mathrm{for~all}\; k \geq k_0$, then the trial point will be accepted eventually because the estimates $\Psi_k^0$ and
    $\Psi_k^{z^k}$ are $\epsilon_F$-accurate. Recall that $\sigma_{k}$ is decreasing in the case of accepting $z^k$. This means that $\sigma_k$ will be bounded above, which leads to a contradiction. Hence, we conclude that $\liminf_{k \rightarrow \infty} \|\grad \Psi(\theta^k)\| = 0$ will hold almost surely.
\end{proof}
\begin{Remark}
	Analogous to \cite[Theorem 4.18]{chen2018stochastic}, one can show that $\lim\limits_{k\rightarrow \infty} \|\grad \Psi(\theta^k)\| \\ = 0$ will hold almost surely by assuming the Lipschitz continuity of $\grad \Psi$.
\end{Remark}

\subsection{Convergence rate analysis of RNGD} \label{sub:convergence-rate}
In this subsection, we study the local convergence rate of a deterministic version of the RNGD method. To begin, let us write $L(z,y) := -\log p(y|z)$ and suppose that $P_x$ is the empirical distribution defined by $\mathcal{S}$. Then, according to the definition of RFIM in \eqref{eq:RFIM} and the chain rule, we have
\[ F^R(\theta) = \frac{1}{|\mathcal{S}_x|} \sum_{x\in \mathcal{S}_x} J^R(x, \theta)^\top F_L(x, \theta) J^R(x,\theta),   \]
where $\mathcal{S}_x:=\{ x : (x,y) \in \mathcal{S} \}$,  $F_L(x, \theta) := \E_{P_{y|x}(\theta)} [ \nabla_z \log p(y|z) \nabla_z \log p(y|z)^\top ]|_{z = f(x, \theta)}$, and $J^R(x, \theta) := [\grad f_1(x,\theta), \ldots, \grad f_q(x,\theta)]^\top$ is the Riemannian Jacobian of $f(x, \theta) \\ =[f_1(x, \theta), \ldots, f_q(x, \theta)]^\top$ with respect to $\theta$. Throughout this subsection, we make the following assumptions on the loss function $L$. 
\begin{assumption}	\label{assum:L}
	For any $y \in \mathcal{S}_y := \{ y : (x,y) \in \mathcal{S} \}$, the loss function $L(\cdot, y)$ is smooth and $\mu$-strongly convex and has $\kappa_L$-Lipschitz gradient and $\kappa_H$-Lipschitz Hessian, namely, 
	\[ \mu I \preceq \nabla_{zz}^2 L(z,y) \preceq \kappa_L I, \quad \| \nabla_{zz}^2 L(z,y) - \nabla_{zz}^2 L(x,y)  \| \leq \kappa_H \|z - x\|, \quad \forall z, x \in \R^n. \]
	%	\be \label{eq:hess-Lsmooth} \| \nabla L(z,y) - \nabla L(x,y)  \| \leq \kappa_L \|z - x\|, \ee
	%	and
	%	\be \label{eq:hess-mustrongconvex} L(x,y) \geq L(z, y) + \nabla L(z, y) + \frac{\mu}{2}\|z - x\|^2, \;\; \forall x,z \in \mathcal{Y},   \ee
	%	and 
	%	\be \label{eq:hess-HLip} \| \nabla^2 L(z,y) - \nabla^2 L(x,y)  \| \leq \kappa_L \|z - x\|. \ee
	In addition, the following condition holds:
	\be \label{eq:fim-gn} F_L(x,\theta) = \nabla_{zz}^2 L(z,y) |_{z=f(x,\theta)} := H_L(f(x,\theta)). \ee 
\end{assumption} 

We remark that the equality \eqref{eq:fim-gn} holds if $\nabla_{zz}^2 L(z,y)|_{z=f(x,\theta)}$ does not depend on $y$, which is the case for the square loss $L(z,y) = \| z-y \|^2$ and the cross-entropy loss $L(y,z) = -\sum_{j} y_j \log z_j$. We refer the reader to \cite[Section 9.2]{martens2014new} for other loss functions that satisfy \eqref{eq:fim-gn}. We remark that the square loss $L(z,y) = \|z-y\|^2$, which appears in both the LRMC and low-dimension subspace learning problems, satisfies Assumption \ref{assum:L}.

Now, we write $\mathcal{S} = \{ (x_i,y_i) \}_{i=1}^N$ with $N = |\mathcal{S}|$ and $u(\theta) = [f(x_1,\theta), \ldots, f(x_N,\theta)]^\top$. Define $J^R(\theta) := [J^R(x_1, \theta),$ $ \ldots, J^R(x_N, \theta)]$ and  $H_L(u(\theta)) := \mathrm{blkdiag}(H_L(u(\theta)_1), \ldots,$ $ H_L(u(\theta)_N))$. 
Then, we have $F^R(\theta) = J^R(\theta)^\top H_L(u(\theta)) J^R(\theta)$. For simplicity, let $u^k := u(\theta^k)$.
Note that $F^R(\theta)$ may be singular when $J^R(\theta)$ is not of full
column rank. In this case, provided that $\left( J^R(\theta^k) J^R(\theta^k)^\top \right)^{-1}$ exists, we can use the pseudo-inverse 
 \[ F^R(\theta^k)^\dagger =
J^R(\theta^k)^\top (J^R(\theta^k)J^R(\theta^k)^\top)^{-1}H_L(u^k)^{-1}
(J^R(\theta^k)J^R(\theta^k)^\top)^{-1}J^R(\theta^k) \] for computation. As mentioned at the beginning of this subsection, we focus on a deterministic version of the RNGD method, in which we adopt a fixed step size $t > 0$ and perform the update 
\be \label{eq:rngd1}
d^k = (F^R(\theta^k))^{\dagger} J^R(x, \theta^k)^\top \nabla L(u^k, y), \quad \theta^{k+1} = R_{\theta^k}(-t d^k)).
% \mathrm{~with~}.
\ee
For concreteness, let us take $R$ to be the exponential map for $\mathcal{M}$ in our subsequent development. Our convergence rate analysis of this deterministic RNGD method can be divided into two steps. The first step is to prove that the iterates $\{\theta^k\}$ always stay in a neighborhood of $\theta^0$ if $J^R$ satisfies certain stability condition. The second step is to establish the convergence rate of the method by utilizing the strong convexity of $L$.  Motivated by \cite{zhang2019fast}, we now formulate the aforementioned stability condition on $J^R$. 
\begin{assumption}\label{con:stableJ}
	For any $\theta$ satisfying $\|\theta-\theta^0\|\le 4 \kappa_L(\mu\sigma_0)^{-1} \|u^0-y\|$, where $\sigma_0 := \sqrt{\lambda_{\min}(J^R(\theta^0) J^R(\theta^0)^\top)} > 0$, it holds that
	\begin{equation} \label{eq:con-stableJ}
		\|J^R(\theta)-J^R(\theta^0)\| \le \min\left\{\frac{1}{2},\frac{\mu}{6\kappa_L }\right\}\sigma_0.
	\end{equation}
\end{assumption}

As will be seen in Section \ref{subsec:jacobian-2layer}, Assumption \ref{con:stableJ} is satisfied by the Riemannian Jacobian that arises in a two-layer fully connected neural network with batch normalization and sufficiently large width. We are now ready to prove the following theorem.
\begin{theorem}\label{thm:4.2.1}
	Let $R$ be the exponential map for $\mathcal{M}$. Suppose that Assumptions \ref{assum:L} and \ref{con:stableJ} hold. Let $\{\theta^k\}$ be the iterates generated by \eqref{eq:rngd1}. 
	\begin{itemize}
		\item[{\rm{(a)}}] There exists a constant $\kappa_R > 0$ such that if $\|u^0 - y\| < \frac{\mu}{3\kappa_H}$ and $t \leq \min \left\{ 1,\left(\frac{1}{6|\|u^0 - y\|} - \frac{\kappa_H}{2\mu} \right)\cdot \frac{3 \mu^2 \sigma_0} {8 \kappa_R\kappa_L^2}\right\} $, then
		\begin{equation} \label{eq:linear-uk}
			\| u^{k+1} - y \| \le \left(1-\frac{t}{2} \right) \|u^k-y\|.
		\end{equation}
		
		\item[{\rm(b)}] Suppose further that $J^R$ is $\kappa_J$-Lipschitz continuous with respect to $\theta$, i.e.,
		\be \label{eq:Lipschitz-Jacobian} \| J^R(\theta) - J^R(\nu) \| \leq \kappa_J \| \theta - \nu\|, \quad \forall \theta, \nu \in \mathbb{R}^r. \ee
		The rate of convergence is quadratic when $t=1$, namely, there is a constant $\kappa_q > 0$ such that 
		\begin{equation}
			\| u^{k+1}-y\| \le \kappa_q \| u^k-y\|^2.
		\end{equation}
	\end{itemize}
\end{theorem}
\begin{proof}
	(a). We proceed by induction. Assume that for $j \leq k$, we have 
	\[ \|\theta^j - \theta^0\| \leq   4 \kappa_L(\mu\sigma_0)^{-1} \|u^0-y\|, \;\; \|u^j - y\| \leq \left( 1 - \frac{\eta}{2} \right) \|u^{j-1} - y\|. \]
	By the definition of $d^k$ in \eqref{eq:rngd1},
	%	 $d^k = -F^R(\theta^k)^\dagger J^R(\theta^k)^\top\nabla_{u}L(u^k,y) = -J^R(\theta^k)^\top(J^R(\theta^k)$ $J^R(\theta^k)^\top)^{-1}H_k^{-1}\nabla_{u}L(u^k,y) =-J^R(\theta^k)^\top(J^R(\theta^k)J^R(\theta^k)^\top)^{-1}H_k^{-1}(\nabla_{u}L(u^k,y)$   
	%	 $-\nabla_{u}L(y,y) ) $. The second equation holds because $y$ is the global minimal point of $L$ and thus $\nabla_{u}L(y,y) = 0$. Using the Lipshitz continuity of $L$ we get 
	\begin{equation}\label{eqn:dk}
		\begin{aligned}
			\|d^k\|&\le \|J^R(\theta^k)^\top(J^R(\theta^k)J^R(\theta^k)^\top)^{-1}\|\|H_L(\theta^k)^{-1}\| \|\nabla_{u}L(u^k,y)-\nabla_{u}L(y,y)\| \\
			&\le \mu^{-1}\kappa_L\sigma_{\min}^{-1}(J^R(\theta^k))\|u^k-y\| \\
			& \le 2\kappa_L(\mu\sigma_0)^{-1}\|u^k-y\|,
		\end{aligned}
	\end{equation}
	where the first inequality is due to $\nabla L(y,y) = 0$ and the last inequality is from Assumption \ref{con:stableJ}. Now, define the map $c_k : [0,1] \rightarrow \mathcal{M}$ as $c_k(s) = R_{\theta^k}(- s t d^k)$. Note that for the exponential map $R$, the geodesic distance between $\theta$ and $R_{\theta}(\xi)$ is equal to $\|\xi\|$ \cite[Equation (7.25)]{opt-manifold-book}, and inequality (2.2) holds with $\alpha = 1$ when we take the Euclidean metric as the Riemannian metric on $\mathcal{M}$. Thus, for any $s \in [0,1]$,
	\[ \begin{aligned}
		\|c_k(s) - \theta^0\| & \leq \| c_k(s) - \theta^k \| + \sum_{j=0}^{k-1}\|\theta^{j+1} - \theta^j \| \leq  t \sum_{j=0}^{k}\|d^j\| \\ & \le2 \kappa_L(\mu\sigma_0)^{-1} t\sum_{j=0}^{k}\|u^j-y\|,
	\end{aligned}
	 \]
where the second inequality is due to \eqref{eq:bound-retraction1}. Since $\|u^j-y\| \leq (1 - \frac{\eta}{2})\|u^{j-1} - y\|$ for all $j \leq k$, we have $\|c_k(s) - \theta^0\| \leq 4 \kappa_L(\mu\sigma_0)^{-1} \|u^0-y\| \; \mathrm{for~all~} s \in (0,1]$. This gives $\|\theta^{k+1} - \theta^0 \| \leq 4\mu \kappa_L\sigma_0^{-1} \|u^0-y\|$.
	To prove \eqref{eq:linear-uk}, we split $\|u^{k+1} - y\|$ into three terms, namely,
	\be \label{eq:est-uk}
	\begin{aligned}
		u^{k+1}-y  = & u^{k+1} - u^{k} + u^k - y  = \int_0^1 J^R(c_k(s))c_k^\prime(s){\rm d}s+u^k-y \\
		=&  \underbrace{\int_0^1 J^R(c_k(s))(c'_k(s)- t d^k){\rm d}s}_{b_1} + \underbrace{t\int_0^1 (J^R(c_k(s))-J^R(\theta^k))d^k{\rm d}s}_{b_2} \\ &+ \underbrace{t\int_0^1 J^R(\theta^k)d^k{\rm d}s+u^k-y}_{b_3}.
	\end{aligned}
	\ee
	For the exponential map $R$ \cite[Equation (5.24)]{opt-manifold-book}, it holds that
	\begin{equation} \label{eq:2nd-order-retraction} 
		c'_k(s) - t d^k = c''_k(s)[-s t d^k] + \tilde{\kappa}_R s^2 t^2 \|d^k\|^2,
	\end{equation}
	where $c''_k(s)[- s t d^k]$ belongs to the normal space to $\mathcal{M}$ at $c_k(s)$ and $\tilde{\kappa}_R > 0$ is the smoothness constant. Plugging \eqref{eq:2nd-order-retraction} into \eqref{eq:est-uk}, we have
	\[ \begin{aligned}
		\|b_1\| & \leq \int_{0}^1 (\|J^R(\theta^0)\| + \|J^R(c_k(s)) - J^R(\theta^0)\|) \tilde{\kappa}_R s^2 t^2 \|d^k\|^2 {\rm d}s \\
		& \leq \int_{0}^1 2\sigma_0 \kappa_R  s^2 t^2 \|d^k\|^2 {\rm d}s = \frac{2}{3}\sigma_0 \kappa_R t^2 \|d^k\|^2,
	\end{aligned}
	 \]
	 where $\kappa_R:= \tilde{\kappa}_R \cdot (1/4 +  \|J^R(\theta^0)\|/(2\sigma_0))$. By \eqref{eq:con-stableJ} and \eqref{eqn:dk}, we have
	\[ \|b_2\| \leq t \int_{0}^1 \min\left\{\frac{1}{2},\frac{\mu}{6\kappa_L }\right\}\sigma_0 \cdot  2\kappa_L(\mu\sigma_0)^{-1}\|u^k-y\| {\rm d}s \leq \frac{t}{3}\|u^k - y\|. \]
	Now, the update \eqref{eq:rngd1} yields $J^{R}\left(u^{k}\right)
    d^{k}=H_{L}\left(u^{k}\right)^{-1} \nabla L\left(u^{k}, y\right)$. It follows that
	\bee
	\begin{aligned}
	\|	b_{3}\| &=\|u^{k}-y- t H_{L}\left(u^{k}\right)^{-1}\left(\nabla L\left(u^{k}, y\right)-\nabla L(y, y)\right) \|\\
		&=\|H_{L}\left(u^{k}\right)^{-1}\left(H_{L}\left(u^{k}\right)\left(u^{k}-y\right)-t\left(\nabla L\left(u^{k}, y\right)-\nabla L(y, y)\right)\right) \|\\
		&=\left\| H_{L}\left(u^{k}\right)^{-1}\left(H_{L}\left(u^{k}\right)\left(u^{k}-y\right)-t \int_{0}^{1} H_{L}\left(u^{k}+s\left(y-u^{k}\right)\right)\left(u^{k}-y\right) {\rm d} s \right) \right\| \\
		&=\left \|H_{L}\left(u^{k}\right)^{-1}\left[\int_{0}^{1}\left(H_{L}\left(u^{k}\right)-t H_{L}\left(u^{k}+s\left(y-u^{k}\right)\right)\right) {\rm d} s\right]\left(u^{k}-y\right) \right\| \\
		& \leq \int_{0}^{1}\left(1-t+t \mu^{-1} \kappa_{H} s\left\|u^{k}-y\right\|\right) {\rm d} s \cdot\left\|u^{k}-y\right\| \\
		& = \left(1-t+\frac{\kappa_{H} t}{2 \mu}\left\|u^{k}-y\right\|\right)\left\|u^{k}-y\right\|,
	\end{aligned}
	\eee
	where the first inequality is due to Assumption \ref{assum:L}. 
	% \[\begin{aligned}
	% 	\left \|\int_0^1 (J^R(c_k(t))-J^R(\theta^k))d^k dt \right \| \le2\mu\kappa_L\sigma_0^{-1} \|J^R(c_k(t))-J^R(\theta^k)\|_2\|u^k-y\|\le\frac{1}{3}\|u^k-y\| \end{aligned} \]
	%  by Condition \ref{con:4.2.1} of $J^R$. Write $u^{k+1}-y$ as $ u^{k+1}-u^k+u^k-y= \int_0^1 J^R(c_k(t))c_k^\prime(t)dt+u^k-y = \int_0^1 J^R(c_k(t))(c_k(t)-\eta d^k)dt +\eta\int_0^1 (J^R(c_k(t))-J^R(\theta^k))d^kdt +\eta\int_0^1 J^R(\theta^k)d^kdt+u^k-y.$ Now we focus on estimating the term $u^k-y+\eta\int_0^1J^R(\theta^k)d^kdt=u^k-y+\eta \JK d^k.$ 
	% By strong convexity of $L$ we have $ \|H_k^{-1}(H_k-\eta\nabla_{u}^2L(u^k+t(y-u^k)) \|_2=\| I-\eta H_k^{-1}\nabla_{u}^2L(u^k+t(y-u^k)\|_2 = \|(1-\eta)I-\eta H_k^{-1}(H_k-\nabla_{u}^2L(u^k+t(y-u^k)))\| \le (1-\eta)+t\eta\kappa_H\|H_k^{-1}\|_2\|u^k-y\|\le1-\eta+\eta\mu\kappa_H\|u^k-y\|.$ 
	Combining the estimates on $b_1, b_2, b_3$, we conclude that
	\begin{equation}\label{eq:4.2.1a}
		\begin{aligned}
			\|u^{k+1}-y\| 
			&\le \left(1-\frac{2t}{3}+\frac{\kappa_H t}{2\mu} \|u^k - y\|\right)\|u^k-y\|+\frac{8}{3}\mu^{-2} \kappa_R\kappa_L^2\sigma_0^{-1}t^2\|u^k-y\|^2 \\
			&\le \left(1-\frac{t}{2} \right)\|u^k-y\|
		\end{aligned}
	\end{equation}
	whenever $\|u^k-y\| < \frac{\mu}{3\kappa_H}$ and $t \leq \left(\frac{1}{6\|u^k - y\|} - \frac{\kappa_H}{2\mu} \right) \cdot \frac{3 \mu^2 \sigma_0} {8 \kappa_R\kappa_L^2}$. Therefore, the inequality \eqref{eq:linear-uk} holds by using the inductive hypothesis $\|u^k-y\|\le\|u^0-y\|$.
	%Here the first and second inequality comes from Lemma \ref{lem:4.2.1} and \eqref{eqn:dk} while the third equation derives from the fact that $\|u^k-y\|$ is small enough such that the norm of the second-order term $(\eta\mu\kappa_H+4\mu^2\kappa_R\kappa_L^2\sigma_{min}(J^R(\theta^0))^{-2}\eta^2)\|u^k-y\|^2\le \frac{\eta}{6}\|u^k-y\|$
	%when 
	
%	The remaining part is to verify the distance between $\theta^k$ and $\theta^0$ is small enough in each iteration. Note that
%	\[ \begin{aligned}
%		\|\theta^{k+1}-\theta^0\| \leq  \sum_{j=0}^{k}\|\theta^{j+1}-\theta^j\|\le\alpha\eta\sum_{j=0}^{k}\|d^j\| \le2\alpha\kappa_L(\mu\sigma_0)^{-1}\eta\sum_{j=0}^{k}\|u^j-y\|, 
%	\end{aligned}\]
%	where the second inequality is due to \eqref{eq:bound-retraction1}. Since $\|u^j-y\| \leq (1 - \frac{\eta}{2})\|u^{j-1} - y\|$ for all $j \leq k$, $\|\theta^{k+1} - \theta^0\| \leq 4\alpha\kappa_L(\mu\sigma_0)^{-1} \|u^0-y\|$. Therefore the iterates $\{\theta^k\}$ lie in the neighborhood mentioned in Condition \ref{con:stableJ}. The proof of (a) is completed.
	
	(b). The proof is similar to that for (a). Substituting $t = 1$ into  \eqref{eq:est-uk}, we obtain 
	%\begin{equation}
	\[ \begin{aligned}
		\|u^{k+1}-y\| & \le 
		\frac{\kappa_H}{2\mu}\|u^k-y\|^2 + \frac{1}{2}\kappa_J   \| d^k \|^2 +\frac{8}{3}\mu^{-2} \kappa_R\kappa_L^2\sigma_0^{-1} \|u^k-y\|^2 \\
		& \leq \left[\frac{\kappa_H}{2\mu} +2\kappa_L^2(\mu\sigma_0)^{-2} \left(\kappa_J  +\frac{4}{3}\sigma_0 \kappa_R\right)\right]\|u^k-y\|^2,
	\end{aligned} \]
	where we use \eqref{eq:Lipschitz-Jacobian} to get 
	\[
	\begin{aligned}
&	\left\|\int_0^1 (J^R(c_k(s))-J^R(\theta^k))d^k{\rm d}s \right \|
\\\le&\kappa_J\int_0^1\|c_k(s)-\theta^k\|\|d^k\|{\rm d}s	\le\frac{1}{2}\kappa_J\|d^k\|^2\le2\kappa_J\kappa_L^2 (\mu\sigma_0)^{-2}\|u^k-y\|^2. 
	\end{aligned}	\]
	%	Thus 
	%	\begin{equation}
	%		\|u^{k+1}-y\| \le \left[\mu\kappa_H+4\mu^2\kappa_L^2\sigma_{min}(J^R(\theta^0))^{-2}(1+\kappa_R)\right]\|u^k-y\|^2.
	%	\end{equation}
	%	Therefore, 
	%	for $\|u^0-y\|$ small enough the rate of convergence is quadratic. 
	The verification of the neighborhood condition for $\theta^k$ is similar to that in (a). This completes the proof.
	%\end{equation}
\end{proof}

\subsection{Jacobian stability of two-layer neural network with batch normalization} \label{subsec:jacobian-2layer}
From the previous subsection, we see that the Jacobian stability condition in Assumption \ref{con:stableJ} plays an important role in the convergence rate analysis of the RNGD method. Let us now show that such a condition is satisfied by a two-layer neural network with batch normalization, thereby demonstrating its relevance. The difference between our setting and that of \cite{zhang2019fast} lies in the use of batch normalization. To begin, consider the input-output map $f$ given by 
\begin{equation}
	f(x, \theta, a)=\frac{1}{\sqrt{m}} \sum_{j=1}^{m} a_{j} \phi\left(\frac{\theta_{j}^{\top} (x-\mathbb{E}[x])}{\sqrt{\theta_j^\top V \theta_j}}\right),
\end{equation}
where $x\in\R^{n}$ is the (random) input vector, $V = \mathbb{E}[ (x - \mathbb{E}[x]) (x - \mathbb{E}[x])^\top ]$ is the covariance matrix,
% $w=[w_1^{\top},w_2^{\top},\dots,w_m^{\top}]^{\top}$ 
$\theta=[\theta_1^\top, \theta_2^\top, \dots, \theta_m^\top]^\top \in \R^{mn}$ is the weight vector
of the first layer, $a_j\in\R$ is the output weight of
hidden unit $j$, and $\phi$ is the ReLU activation function. This represents a
single-output two-layer neural network with batch normalization. We fix the $a_j$'s throughout as in \cite{zhang2019fast} and apply the RNGD method with
a fixed step size on $\theta$, in which each weight vector $\theta_j$ is assumed to be normalized. 
For the Grassmann manifold ${\rm Gr}(1,n)$, we choose $d$ with $\|d\|=1$ as the representative element of the one-dimensional subspace $\{ cd : c \ne 0 \}$. With a slight abuse of notation, we write ${\rm Gr}(1,n) := \{d \in \R^n : \|d\| = 1 \}$. Then, we can regard the vector $\theta$ as lying on a Cartesian product of $\mbox{Gr}(1,n)$'s.
%With a slight abuse of notations, we denote ${\rm Gr}(1,n) := \{d \in \R^n : \|d\| = 1 \}$ as one representative class of Grassmannian manifolds. The exponential map is used as retraction to maintain the unit length during the iteration.

It is well known that if $\theta_j$ is a standard Gaussian random vector, then the random vector $\theta_j / \|\theta_j\|$ is uniformly distributed on ${\rm Gr}(1,n)$. 
We draw each $\theta_j$ uniformly from ${\rm Gr}(1,n)$ and each $a_j$ uniformly from $\{-1,+1\}$. As mentioned in Section \ref{sec:3.2}, we have $J^R(\theta) = J(\theta)$. Thus, our goal now is to establish the stability of $J$.
%\begin{remark}\label{rmk:con4.31}
%	Different with the Gaussian distribution-type initialization \cite{zhang2019fast,wu2019global,du2018gradient}, we initialize $w_r$ uniformly distributed on the surface of unit ball. 
%	This is because their two-layer neural networks did not have Batch Normalization, while it exists in our network structure since we are optimizing on Grassmannian manifold. It is easy to see that by just normalizing to the unit length, one can transform the Gaussian distribution with zero mean to uniformly distribution on $B_1$. Moreover, in the case of BN, as mentioned in \ref{sec:3.2}, we can see $J^R(\theta) = J(\theta)$, thus it is enough to prove the smoothness of $J$.
%\end{remark}
To begin, let $\mathcal{S}=\{ (x_i,y_i) \}_{i=1}^N$ denote the dataset and $u(\theta)=[f(x_1, \theta,a),$ $f(x_2, \theta,a),\dots,f(x_N,\theta,a)]^\top$ denote the output vector. Following \cite{du2018gradient,wu2019global,zhang2019fast}, we make the following assumption on $\mathcal{S}$.
\begin{assumption}\label{cond:con3}
	 For any $(x,y) \in \mathcal{S}$, it holds that $\|x\| =1 \text
     { and }\left|y\right|=\mathcal{O}(1)$.  For any $x_i, x_j \in \mathcal{S}_x$ with  $ i \neq j,$ it holds that
     $  x_i\neq \pm x_j .$ In addition, the input vector $x$ satisfies $\mathbb{E}[x] = {\bf 0}$ and the covariance matrix $V = \mathbb{E}[ xx^\top ]$ is positive definite with minimum eigenvalue $\sigma_V > 0$.
\end{assumption}

%Following \cite{du2018gradient,wu2019global,zhang2019fast}, we only optimize the weights of the first layer and fix the second layer just for simplicity. Based on the similar analysis, we can also prove the convergence for both layers. Therefore, the RNGD can be simplified to
%\begin{equation}
%	w^{k+1} = R_{\theta^k}(-\eta(F^R)^\dagger\grad f(w^k)(u^k-y)).
%\end{equation}

%Then define the limiting Gram matrix as
%\begin{definition}[Limiting Gram Matrix]\label{def:con1}
%	The limiting Gram matrix $G^\infty\in\R^{n\times n}$is defined as
%	\begin{equation}\label{eqn:gram}
%		G_{i j}^{\infty}=\mathbb{E}_{w \sim \mathcal{N}\left(0, ^{2} I\right)}\left[x_{i}^{\top} x_{j} I\left\{w^{\top} x_{i} \geq 0, w^{\top} x_{j} \geq 0\right\}\right]=x_{i}^{\top} x_{j} \frac{\pi-\arccos \left(x_{i}^{\top} x_{j}\right)}{2 \pi}.	
%	\end{equation}
%\end{definition}
%It is proved positive definite in \cite{du2018gradient}. Let $\lambda_0 = \ldm(G^\infty) > 0 $.$G(t) = J(t)J(t)^\top$ is defined similarly as $G_{i j}(t)=\frac{1}{m} x_{i}^{\top} x_{j} \sum_{r=1}^m I\left\{w_{r}(t)^{\top} x_{i} \geq 0, w_{r}(t)^{\top} x_{j} \geq 0\right\}.$ 

Motivated by \cite{zhang2019fast}, we use $[x_i^\top \theta_j^0]_{k-}$ to represent the $k$-th smallest entry of $[x_i^\top \theta_1^0, \\x_i^\top \theta_2^0,\dots,x_i^\top \theta_m^0]$ in absolute value. Since $V$ is positive definite and $\mbox{Gr}(1,n) = \{ d \in \mathbb{R}^n : \| d \| = 1 \}$ is compact, for $i=1,\ldots,N$, the function $u \mapsto \varphi_i(u) = \frac{x_i}{\sqrt{u^\top V u}} - \frac{V uu^\top x_i}{(u^\top Vu)^{3/2}} $ is $L$-Lipschitz on ${\rm Gr}(1,n)$ for some constant $L >0$, i.e.,  $\|\varphi_i(u)-\varphi_i(v)\| \le L\|u-v\|$ for any $u, v \in \mbox{Gr}(1,n)$. 
% it then can be proved with a similar process in \eqref{eq:est-uk}.) 
 To prove the desired Jacobian stability result, we need the following lemmas. They extend those in \cite{zhang2019fast}, which are developed for the Euclidean setting, to the Grassmann manifold setting. In what follows, we use $\delta_{A}$ to denote the indicator function of an event $A$, i.e., $\delta_{A}$ takes the value $1$ if the event $A$ happens and $0$ otherwise.
\begin{lemma}\label{lem:con4.1}
	Let $\theta_j, \theta_j^0 \in {\rm Gr}(1,n)$, where $j=1,\ldots,m$, be given. Suppose that for some $k \in \{1,\ldots,m\}$, we have $\left\| \theta-\theta^0\right\| \le\sqrt{k}[x_i^\top \theta_j^0]_{k-}$ for $i=1,2,\dots,N$ and $j=1,2,\dots,m$. Then, we have
	\begin{equation}
		\left\| J(\theta)-J(\theta^0)\right\|^2 \le \frac{2NkM + NkL}{m},
	\end{equation}
	where $M = \max_{i \in \{1,\ldots,N\}}\left(\max_{u\in {\rm Gr}(1,n)}\left\|\frac{x_i}{\sqrt{u^\top Vu}} - \frac{Vuu^\top x_i}{(u^\top Vu)^{3/2}}\right\|^2\right)$.
\end{lemma}
\begin{proof}
	Let $A_{i,j}$ denote the event that the signs of $x_i^\top \theta_j$ and $x_i^\top \theta_j^0$ are different.  We claim that, for $i=1,2,\dots,N$, there are at most $2k$ non-zero entries of $\{ \delta_{A_{i,j}} \}_{j=1}^m$. Otherwise, there exists an $i \in \{1,\ldots,N\}$ such that
	\begin{align*}
		&\|\theta-\theta^0\|^2 \ge \sum_{j=1}^m |x_i^\top \theta_j-x_i^\top \theta_j^0|^2 
		\\	\ge &\sum_{j \in \{j : \delta_{A_{i,j}}= 1\} } |x_i^\top \theta_j-x_i^\top \theta_j^0|^2 	\ge \sum_{j \in \{j : \delta_{A_{i,j}}= 1\} } |x_i^\top \theta_j^0|^2 > k [x_i^\top \theta_j^0]_{k-}^2,
	\end{align*}
	which contradicts our assumption. Now, the generalized Jacobian of $f$ with respect to $\theta$ is given by $$J(\theta) = \frac{1}{\sqrt{m}}\sum_{j=1}^m \sum_{i=1}^N a_j \\ \left[ \delta_{x_i^\top \theta_1 \ge 0} \cdot \varphi_i(\theta_1)^\top , \ldots, \delta_{x_i^\top \theta_m \ge 0} \cdot \varphi_i(\theta_m)^\top \right].$$ When $x_i^\top\theta_j$ and $x_i^\top\theta_j^0$ have the same sign, the difference $\delta_{x_i^\top \theta_j \ge 0} \cdot \frac{a_j}{\sqrt{m}}\varphi_i(\theta_j)- \delta_{x_i^\top \theta_j^0 \ge 0} \cdot \frac{a_j}{\sqrt{m}}\varphi_i(\theta_j^0)$ is either $\mathbf{0}$ or $\frac{a_j}{\sqrt{m}}(\varphi_i(\theta_j)-\varphi_i(\theta_j^0))$. Splitting $\|J(\theta)-J(\theta^0)\|^2$ into two parts according to the event $A_{i,j}$ yields
	\begin{align*}
		&	\| J(\theta)-J(\theta^0)\|^2 \\  \le &\frac{M}{m}\sum_{(x_i,y_i)\in \mathcal{S}}\sum_{j=1}^m\delta_{A_{i,j}}+ \frac{L}{m}\sum_{(x_i,y_i)\in \mathcal{S}}\sum_{j=1}^m\|\theta_j-\theta_j^0\|^2 \\
		\le & \frac{2NkM}{m} +\frac{L}{m}\sum_{(x_i,y_i)\in \mathcal{S}}\|\theta-\theta^0\|^2 \\
		\le & \frac{2NkM+NkL}{m},
	\end{align*}
	where the last inequality follows from the assumption on $\| \theta - \theta^0 \|$ and the fact that $|[x_i^\top \theta_j^0]_{k-}|\le 1$ for $i=1,\ldots,N$ and $j=1,\ldots,m$.
\end{proof}

The next lemma gives an upper bound on the probability of the event $\{|x_i^\top \theta_j |\le \gamma\}$ for all $\gamma > 0$, which will be used to estimate $[x_i^\top \theta_j^0]_{k-}$ in Lemma \ref{lem:con4.2}. 
\begin{lemma}\label{lem:con4.32}
	Let $v$ be uniformly distributed on ${\rm Gr}(1,n)$, $x\in {\rm Gr}(1,n)$ be a given unit-norm vector,  and $\gamma>0$ be a given positive number, where $n\ge2$. Then, we have $\mathbb{P}(|x^\top v|\le\gamma)\le \sqrt{\pi n}\gamma$. Moreover, the dependence on $n$ in the bound is optimal up to constant factors.
\end{lemma}
\begin{proof}
	Without loss of generality, we may assume that $x=(1,0,\dots,0)$ since the Euclidean inner product and the distribution of $v$ are invariant under orthogonal transformation. Then, we have $x^\top v = v_1$. Let $Z_1,\dots,Z_n$ be standard Gaussian random variables. Then, the random variable $x^\top v$ has the same distribution as $ B:=\frac{Z_1}{\sqrt{Z_1^2+\dots+Z_n^2}}$. It is well known that $B^2$ follows the distribution $\mbox{Beta}(\frac{1}{2},\frac{n-1}{2})$ \cite[Section 25.2]{johnson1995continuous}. As a result, the density function $h$ of $B$ can be explicitly written as 
	\begin{equation}
		h(r)  = \frac{\Gamma(\frac{n}{2})}{\sqrt{\pi}\Gamma(\frac{n-1}{2})}(1-r^2)^\frac{n-3}{2}, \quad |r|<1.
	\end{equation}
	It follows directly that 
	\begin{equation}
		\begin{aligned}
			\mathbb P(|x^\top v|\le \gamma)=\mathbb P(|B|\le \gamma) =\int_{-\gamma}^{\gamma}h(r)dr\le \frac{\gamma\Gamma(\frac{n}{2})}{\sqrt{\pi}\Gamma(\frac{n-1}{2})}\le \sqrt{\pi n}\gamma,
		\end{aligned}
	\end{equation}
	where the last step uses the classic result $\Gamma(\frac{n}{2})\le \pi\sqrt{n} \Gamma(\frac{n-1}{2})$ in calculus. 

	To see the optimality of the dependence on $n$ in the bound, note that for $\gamma\le \frac{1}{\sqrt{n}}$, we have 
	\[
		\mathbb P(|x^\top v|\le \gamma)=\mathbb P(|B|\le \gamma) =\int_{-\gamma}^{\gamma}h(r)dr\ge \frac{\gamma\Gamma(\frac{n}{2})}{2\sqrt{\pi}\Gamma(\frac{n-1}{2})}\ge \frac{5}{12\sqrt{2e}} \sqrt{n},
	\]
	where the third step uses $(1-r^2)^\frac{n-3}{2}\ge 1-\frac{n-3}{2}r^2$ and the fact that $\gamma \le\frac{1}{\sqrt{n}}$, and the last step follows from an application of Stirling's formula; see, e.g., \cite[Eq. (33)]{so2010}. Hence, the dependence on $n$ in the bound is optimal up to constant factors.
\end{proof}

Using the above lemmas, we show that Assumption \ref{con:stableJ} will hold with high probability.
\begin{lemma}\label{lem:con4.2}
	Let $\theta_j, \theta_j^0 \in {\rm Gr}(1,n)$, where $j=1,\ldots,m$, be given. For any given $Q,\epsilon>0$, if  $\|\theta-\theta^0\|\le Q$, then with probability at least $1-\epsilon$, we will have
	\begin{equation} \label{eq:jacobian-stab}
		\|J(\theta)-J(\theta^0)\|^2\le \frac{2(\pi n)^\frac{1}{3}N^\frac{5}{3}MQ^{\frac{2}{3}}}{\epsilon^\frac{2}{3}m^\frac{1}{3}}+\frac{(\pi n)^\frac{1}{3}N^\frac{5}{3}LQ^{\frac{2}{3}}}{\epsilon^\frac{2}{3}m^\frac{1}{3}}.
	\end{equation}
\end{lemma}

\begin{proof}
	For given integers $k \in \{1,\ldots,m\}$ and $i\in\left\{1,2\dots,N\right\}$, we prove that with probability at least $1-\epsilon/N$, there will be at most $k - 1$ hidden units $\theta_j^0$ such that $|x_i^\top \theta_j^0|\le\frac{k\epsilon}{Nm \sqrt{\pi n}}$. 
	For $\tau>0$, let $\gamma_\tau$ be the positive number such that $\mathbb{P}(|g|\le\gamma_\tau)=\tau$, where $g$ follows the same distribution as $x_i^\top \theta_j^0$. It follows from Lemma \ref{lem:con4.32} that $\gamma_\tau\ge\frac{1}{\sqrt{\pi n}}\tau$. Let $\tau = \frac{k\epsilon}{Nm}$. Then, we have
	\begin{equation}
		\E\left[\sum_{j=1}^{m} \delta_{\left| x_i^\top \theta_j^0\right| \leq \gamma_{\tau}}\right]=\sum_{j=1}^{m} \mathbb{P}\left[\left|x_i^\top \theta_j^0\right| \leq \gamma_{\tau}\right] \leq \frac{k \epsilon}{N}.
	\end{equation}
	Applying the Markov inequality yields
	\begin{equation}
		\mathbb{P}\left[\sum_{j=1}^{m} \delta_{\left|x_i^\top \theta_j^0\right| \leq \gamma_{\tau}} \geq k\right] \leq \frac{\epsilon}{N}.
	\end{equation}
	Therefore, by taking $k=\frac{Q^\frac{2}{3}m^\frac{2}{3}(\pi n)^\frac{1}{3}N^\frac{2}{3}}{\epsilon^\frac{2}{3}}$, the inequalities $\sqrt{k}[x_i^\top \theta^0]_{k-} \ge\frac{k^\frac{3}{2}\epsilon}{Nm\sqrt{\pi n}} = Q$ will hold simultaneously for $i=1,\ldots,N$ with probability at least $1-\epsilon$. The desired conclusion then follows from Lemma \ref{lem:con4.1}.
\end{proof}
%\begin{Remark}
%	Note that the results in \cite[Lemma 7]{zhang2019fast} is about the Gaussian distribution, while ours is ours is about the uniform distribution.  in the case of batch normalization, the order of $N$ is $\frac{5}{3}$.  
%\end{Remark}

With the help of Lemma \ref{lem:con4.2}, we are now ready to establish the convergence rate of the RNGD method when applied to the two-layer neural network with batch normalization.
\begin{theorem}\label{thm:con3}
	 Suppose that Assumptions \ref{assum:L} and \ref{cond:con3}  hold. Let $\epsilon>0$ be a given constant. Suppose that the number $m$ of hidden units satisfy $$m=\Omega\left(\frac{128(L+2M)^3 \pi n  N^{6}\kappa_L^2}{ \mu^2\sigma_{0}^{8} \sigma_V\epsilon^{3} \min\left\{\frac{1}{2}, \frac{\mu}{6\kappa_L} \right \}^6 }\right),$$ where the constants $L, M, \kappa_L, \mu, \sigma_0, \sigma_V$ are defined previously. 
%	 $m=\Omega\left(\frac{(L+2M)^3}{\sigma_{0}^{8}\right.$ $\left.}\cdot\frac{n^2\alpha^{2}N^{6}\kappa_L^5}{\mu^8 \sigma_V\delta^{3}}\right)$
	 If we draw $\theta_j^0$ uniformly from ${\rm Gr}(1,n)$ and $a_j$ uniformly from $\{-1, +1\}$ for $j=1,2\dots,m$, then the Riemannian Jacobian stability condition in Assumption \ref{con:stableJ} will hold with probability at least $1-\epsilon$.
%		\begin{itemize}
%		\item[{\rm{(a)}}]
	Furthermore, when $m\ge \frac{16(L+2M)^3 \pi n N^{5}\kappa_L^2}{ 9 \sigma_{0}^{8} \kappa_H^2 \epsilon^{2} \min\left\{\frac{1}{2}, \frac{\mu}{6\kappa_L} \right \}^6 }$, $\|u^0 - y\| \leq \frac{\mu}{3\kappa_H}$, and $\eta \leq \min\left\{1,\left(\frac{1}{6|\|u^0 - y\|} - \frac{\kappa_H}{2\mu} \right) \cdot \frac{3 \mu^2 \sigma_0} {8 \kappa_R\kappa_L^2}\right\}$, with probability at least $1-\epsilon$, we will have
	\begin{equation} \label{eq:linear}
			\| u^{k+1} - y \| \le \left(1-\frac{1}{2}\eta\right) \|u^k-y\|.
	\end{equation}

%		\item[{\rm(b)}] 
%		The rate of convergence is quadratic for when $\eta=1$, namely, there is a constant $\kappa_q > 0$ such that 
%		\begin{equation}
%			\| u^{k+1}-y\| \le \kappa_q \| u^k-y\|^2.
%		\end{equation}
%	\end{itemize}
\end{theorem}
\begin{proof}
	By Assumption \ref{cond:con3} and the fact that $a_j$ is drawn uniformly from $\{-1,+1\}$, we have $\mathbb{E}\left[u^0\right]=\mathbf{0}$ and 
		\[ \begin{aligned}
		\mathbb{E}\left[(u_j^0)^2\right]& =\mathbb{E}\left[\frac{1}{m}\left(\sum_{j=1}^ma_j\phi\left(\frac{(\theta_j^0)^\top(x-\mathbb{E}[x])}{\sqrt{(\theta_j^0)^\top V\theta_j^0}}\right)\right)^2\right] \\
		& = \mathbb{E}\left[\frac{1}{m}\sum_{j=1}^m\phi\left(\frac{(\theta_j^0)^\top x}{\sqrt{(\theta_j^0)^\top V\theta_j^0}}\right)^2\right] = \mathcal{O}\left(\frac{1}{\sigma_V}\right),\quad j=1,\ldots,N.
		\end{aligned} \]
		This gives
	\begin{equation}
		\mathbb{E}\left[\|u^0-y\|^{2}\right]=\|y\|^{2}+2 y^{\top} \mathbb{E}[u^0]+\mathbb{E}\left[\|u^0\|^{2}\right]=\mathcal{O}\left(\frac{N}{\sigma_V}\right).
	\end{equation}
	Applying the Markov inequality, we see that $\|u^0-y\|^2=\mathcal O\left(\frac{2N}{\epsilon\sigma_V}\right)$ will hold with probability at least $1-\frac{1}{2}\epsilon$.
	This, together with the result of Lemma \ref{lem:con4.2} with $Q = 4 \kappa_L (\mu \sigma_0)^{-1} \| u^0 - y \|$, implies that Assumption \ref{con:stableJ} will hold with probability at least $1-\epsilon$ for $m=\Omega\left(\frac{128(L+2M)^3 \pi n N^{6}\kappa_L^2}{ \mu^2\sigma_{0}^{8} \sigma_V\epsilon^{3} \min\left\{\frac{1}{2}, \frac{\mu}{6\kappa_L} \right \}^6 }\right)$. 
	
	To establish the convergence rate result, observe from Theorem \ref{thm:4.2.1} that $\|\theta^k - \theta^0\| \leq 4 \kappa_L(\mu\sigma_0)^{-1} \|u^0-y\|$ when $\|u^0 - y\| \leq \frac{\mu}{3\kappa_H}$ and $\eta \leq \min\left\{1,\left(\frac{1}{6|\|u^0 - y\|} - \frac{\kappa_H}{2\mu} \right) \cdot \frac{3 \mu^2 \sigma_0} {8 \kappa_R\kappa_L^2}\right\}$. By taking $Q = 4  \kappa_L\sigma_0^{-1}/(3\kappa_H)$ in Lemma \ref{lem:con4.2}, we see that Assumption \ref{con:stableJ} will hold with probability at least $1-\epsilon$ if $m\ge \frac{16(L+2M)^3 \pi n N^{5}\kappa_L^2}{ 9 \sigma_{0}^{8} \kappa_H^2 \epsilon^{2} \min\left\{\frac{1}{2}, \frac{\mu}{6\kappa_L} \right \}^6 }$. Following the proof of Theorem \ref{thm:4.2.1}, we conclude that \eqref{eq:linear} will hold for all $k\ge0$ with probability at least $1-\epsilon$. This completes the proof. 
\end{proof}
\section{Numerical results} \label{sec:numerics}
\subsection{Low-rank matrix completion}
We compare our proposed RNGD method with the Riemannian stochastic gradient
descent (RSGD) method \cite{bonnabel2013stochastic}, the Riemannian stochastic variance-reduced gradient (RSVRG) method \cite{sato2019riemannian}, and the Riemannian conjugate
gradient (RCG) method without preconditioner \cite{boumal2015low}.  All algorithms are initialized by
the QR decomposition of a random $n$-by-$p$ matrix whose entries are generated from the standard Gaussian distribution. We consider two real datasets. One is taken from the Jester joke recommender system,\footnote{The dataset Jester can be downloaded from \url{https://grouplens.org/datasets/jester}} which contains ratings (with scores from $-10.00$ to $+10.00$) of 100 jokes from 24983 users. The other is the movie rating dataset MovieLens-1M,\footnote{The dataset MovieLens-1M can be downloaded from \url{https://grouplens.org/datasets/movielens}} which contains ratings (with stars from $1$ to $5$) of 3952 movies from 6040 users. In the experiments, each dataset is randomly divided into 2 sets, one for training and the other for testing.
%The testing datasets are set by only preserved 
%hslFor RSGD and RSVRG, 
We utilize the implementations of RSGD and RSVRG given in the RSOpt package\footnote{The code of RSOpt can be downloaded from \url{https://github.com/hiroyuki-kasai/RSOpt}} and the implementation of RCG given in the Manopt package.\footnote{The code of Manopt can be downloaded from \url{https://github.com/NicolasBoumal/manopt}} The default parameters therein are used.
%Two versions of RNGD are implemented.  One is RNGD-TR, which uses trust-region policy \ref{alg:RNGD}. 
%For RNGD, we use fixed stepsizes and adopt the variance reduced technique from RSVRG to estimate the stochastic gradient and RFIM. 
For RNGD, the same variance reduction technique as that in RSVRG is adopted to update both the estimated gradient and the approximate RFIM \eqref{eq:rfim-lrmc}. Specifically, we compute $a_i(U)$ for all $i$ in each outer iteration and update $a_i(U)$ if the $i$-th sample is used in the estimation of the gradient. We use fixed step sizes for RNGD and RSVRG. For RSGD, the step size $\eta_k$ is set to $\eta_k = \frac{\eta_0}{1+ \eta_0 k/10}$. We search in the set $\{2, 1, 0.5, \ldots, 2 \times 10^{-8}, 10^{-8},  5\times 10^{-9}\}$ to find the best initial step size $\eta_0$ for RSGD and the best step size for RSVRG. The step size for RNGD is set to $0.05$ for both datasets.
%Default parameters in RSOpt's trustregion method are used for RNGD-TR. 
%We set fixed step size $\eta = 0.05$ for RNGD, while the step size is chosen by a grid search from the set of $1, 0.5, 0.2, 0.1, 0.05, 0.02,$ etc, for both RNGD and baselines, etc. The other dataset is from  The step size is chosen to be $0.05$ for RNGD following the same grid search for both the proposed algorithm and the baselines. 
%In the experiments, the dataset are randomly divided into 2 sets for training and testing as the implementation of RSOpt. 

Figure \ref{fig:lrmr} reports the mean squared error (MSE) on both the training and testing datasets, which are defined as $\| \mathcal{P}_{\Omega_{\rm{train}}}( UA - X ) \|^2 / | \Omega_{{\rm train}}|$ and $\| \mathcal{P}_{\Omega_{\rm{test}}}( UA - X ) \|^2 / | \Omega_{{\rm test}}|$, respectively, where $\Omega_{{\rm train}}$ and $\Omega_{{\rm test}}$ are the sets of known indices in the training and testing datasets, respectively. The label $\#\mathrm{grad}/N$ on the $x$-axis means the number of epochs, which is defined as the number of cycles through the full dataset. We run all algorithms with a specified number of epochs for different datasets. 
We can see that RNGD converges the fastest among the four methods on both datasets. 
\begin{figure}[htb]
	\centering
	\subfigure{\includegraphics[width=0.45\textwidth]{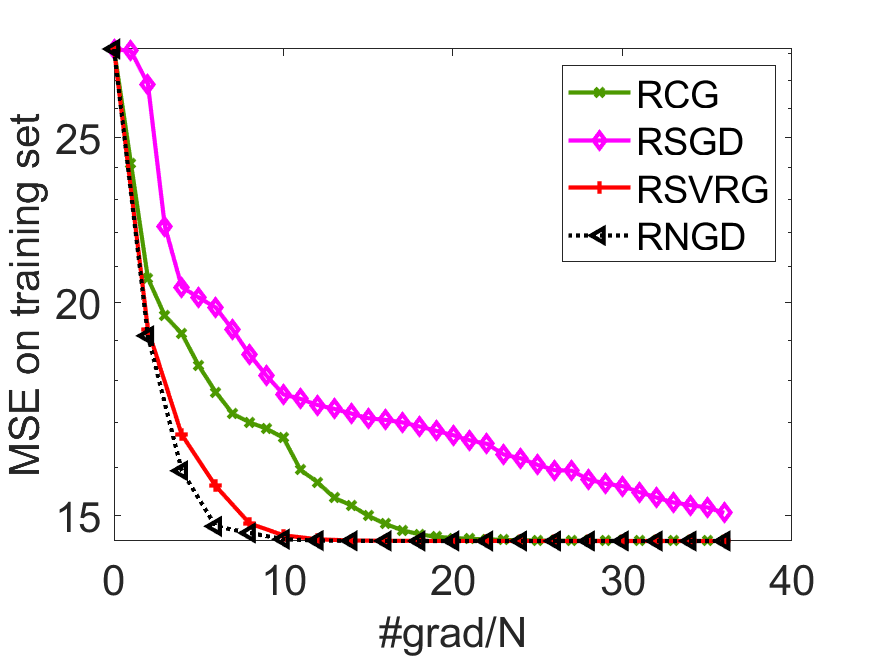}}
	\subfigure{\includegraphics[width=0.45\textwidth]{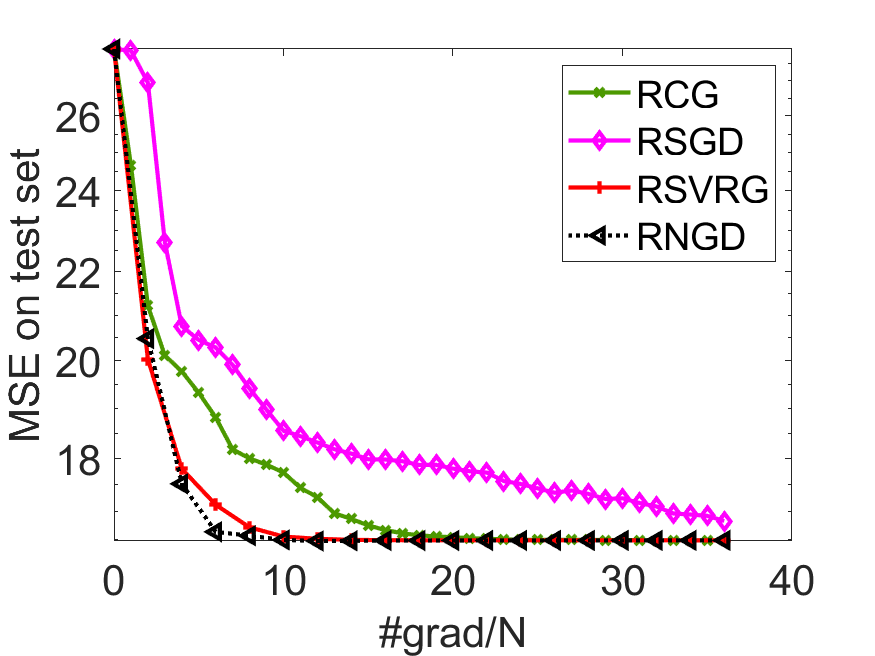}} 
	\subfigure{\includegraphics[width=0.45\textwidth]{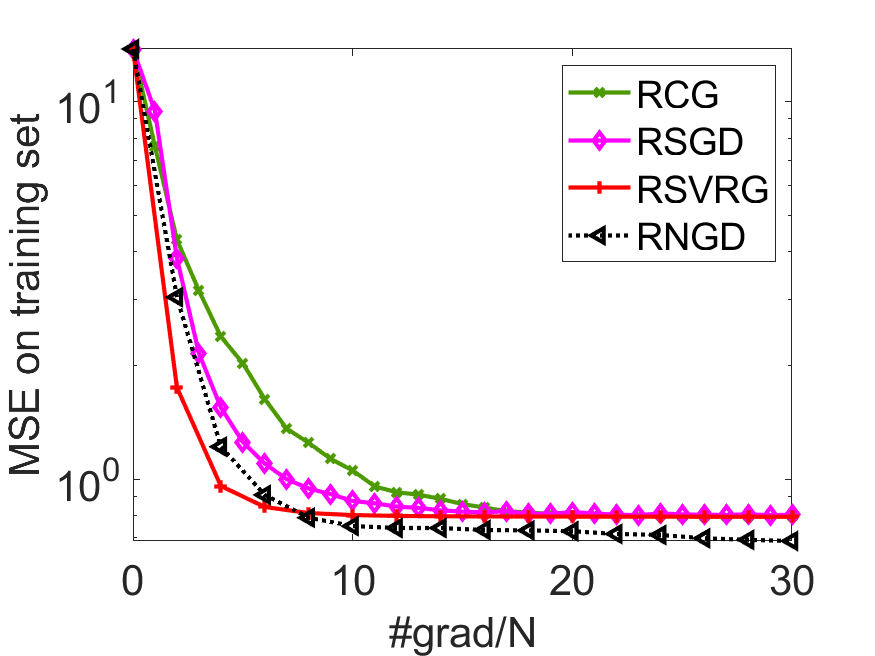}}
	\subfigure{\includegraphics[width=0.45\textwidth]{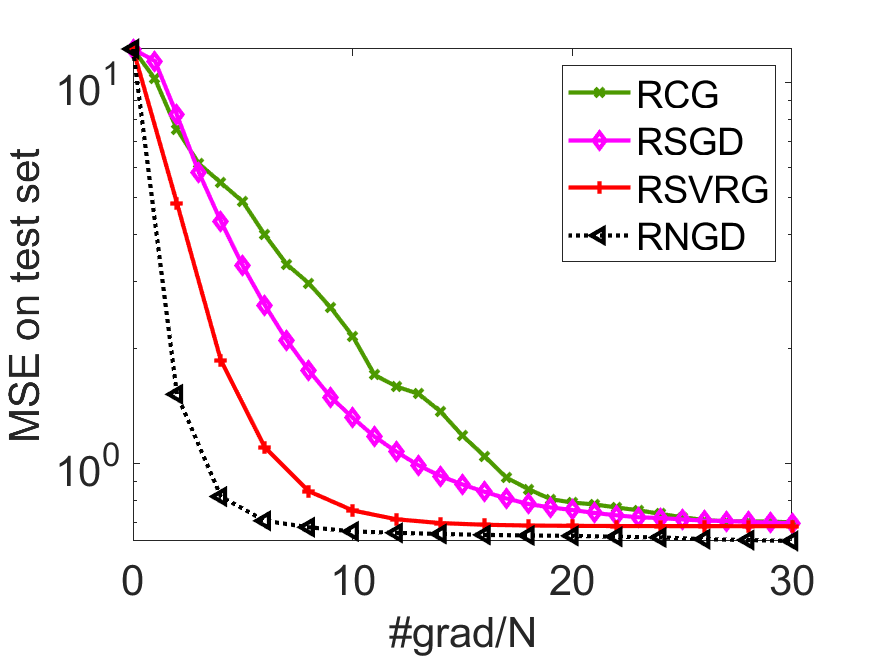}} 
%	\subfigure{\includegraphics[width=0.32\textwidth]{fig/mr/jester_grad.png}} 
	\caption{Numerical results for LRMC on the Jester dataset (first row) and the MovieLens-1M dataset (second row). \label{fig:lrmr} }
\end{figure}

\subsection{Low-dimension subspace learning}

\begin{figure}[htb]
	\centering
%	\subfigure{\includegraphics[width=0.45\textwidth]{fig/sl/90epoch/steps_sgd = 1e-07steps_rngd = 0.02steps_svrg = 1e-06 Means square error on testset per epoch maxepoch = 90random.png}}
%	\subfigure{\includegraphics[width=0.45\textwidth]{fig/sl/90epoch/steps_sgd=1e-07steps_rngd=0.02steps_svrg = 1e-06 Means square relative error on trainset per epoch maxepoch = 90random.png}}\\ 
%	\subfigure{\includegraphics[width=0.45\textwidth]{fig/sl/90epoch/steps_sgd = 1e-07steps_rngd = 0.02steps_svrg = 1e-06 Means square error on testset per epoch maxepoch = 90school.png}} 
%	\subfigure{\includegraphics[width=0.45\textwidth]{fig/sl/90epoch/steps_sgd=1e-07steps_rngd=0.02steps_svrg = 1e-06 Means square error on trainset per epoch maxepoch = 90school.png}} \\
%	\subfigure{\includegraphics[width=0.45\textwidth]{fig/sl/90epoch/steps_sgd = 1e-07steps_rngd = 0.02steps_svrg = 1e-06 Means square error on testset per epoch maxepoch = 90sarcos.png}} 
%	\subfigure{\includegraphics[width=0.45\textwidth]{fig/sl/90epoch/steps_sgd=1e-07steps_rngd=0.02steps_svrg = 1e-06 Means square error on trainset per epoch maxepoch = 90sarcos.png}} \\	
	\subfigure{\includegraphics[width=0.45\textwidth]{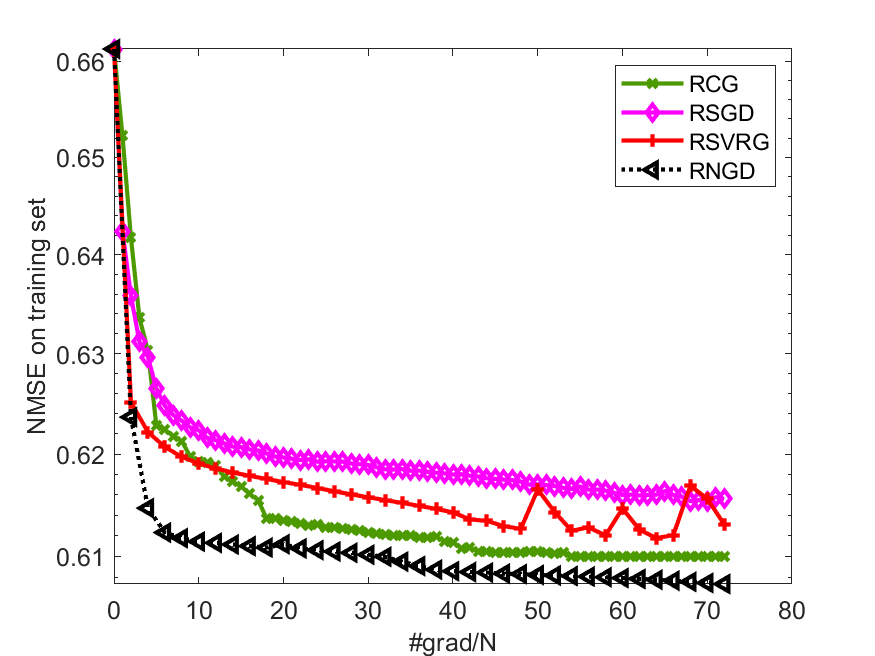}} 
	\subfigure{\includegraphics[width=0.45\textwidth]{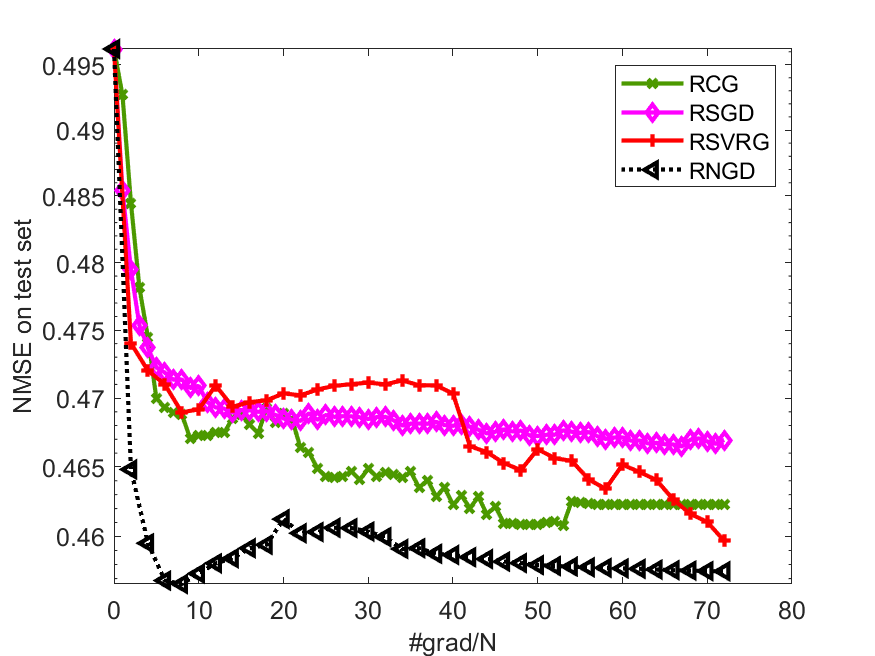}} \\
	\subfigure{\includegraphics[width=0.45\textwidth]{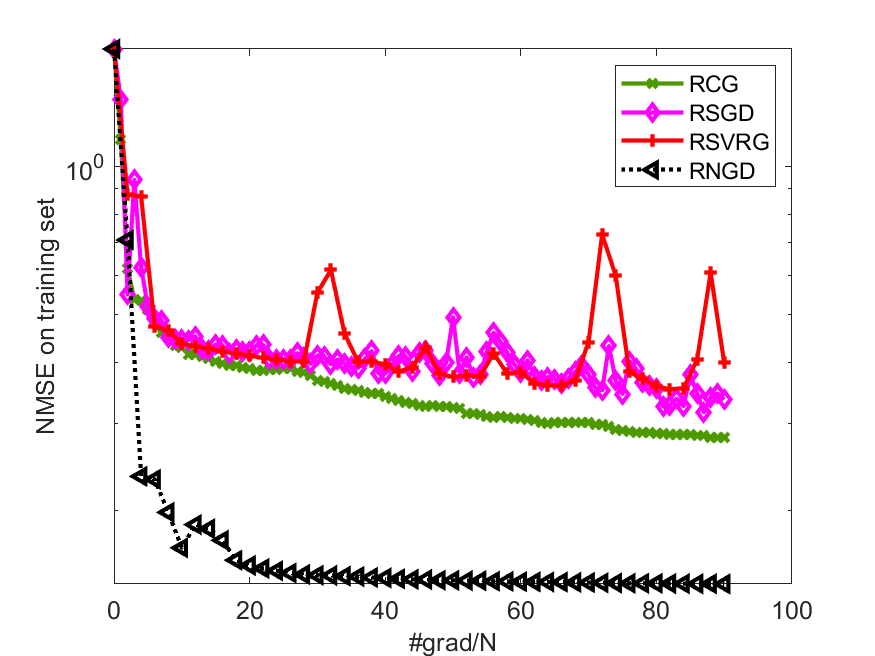}} 
	\subfigure{\includegraphics[width=0.45\textwidth]{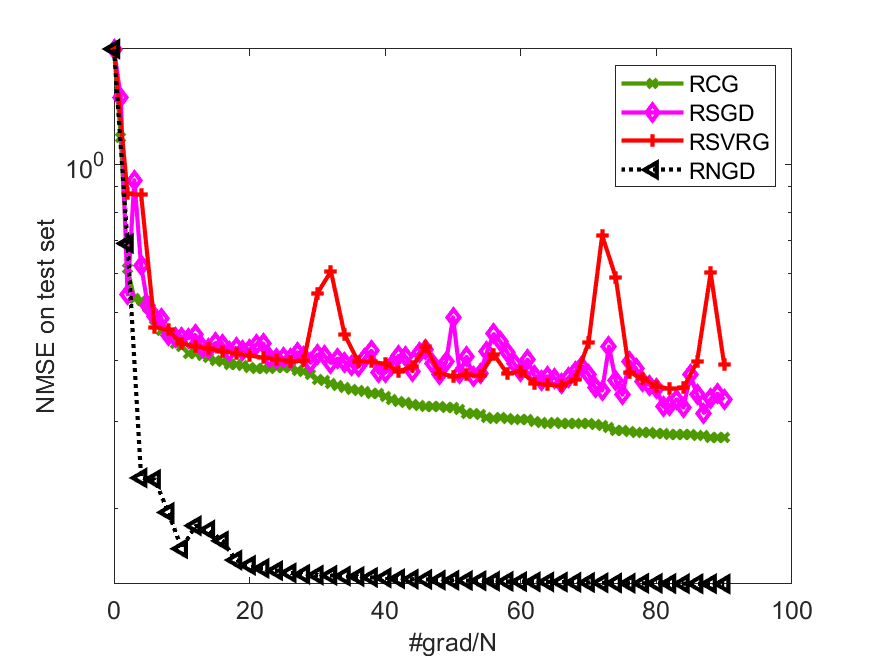}} \\
	
	\caption{Numerical results for multitask learning on the $School$ dataset (first row) and the $Sarcos$ dataset (second row). \label{fig:random} }
\end{figure}

We compare our proposed RNGD with RCG, RSGD, and RSVRG on two real-world datasets: $School$ \cite{goldstein1991multilevel} and $Sarcos$  \cite{vijayakumar2002statistical}. The dimension $p$ is set to be $6$ for both datasets.
% and a synthetic dataset. The synthetic dataset is constructed in the same way as that in \citep{mishra2019riemannian} except that the noise in the label $y_i$ satisfies the standard Gaussian with zero mean and unit variance, which means the noise is larger. 
%For both datasets,
 We choose the best step sizes for RSVRG and RSGD from the set $\{1, 0.5, 0.2, 0.1, 0.05, 0.02, \ldots, 10^{-8}, 5$ $\times  10^{-9}, 2 \times 10^{-9}, 10^{-9}\}$. We use the step size 4 (resp., 1) on the $School$ (resp., {\it Sarcos}) dataset for RNGD. All the codes are implemented within the RSOpt framework and the other parameters of the algorithms are set to the default values therein. 

Figure \ref{fig:random} reports the normalized MSE (NMSE) \cite{mishra2019riemannian} on both datasets, which is the mean of the normalized squared error of all tasks.
For both datasets, RNGD returns a point with the lowest NMSE. Especially for the $Sarcos$ dataset, a significant difference in the NMSE between RNGD and other methods is observed. Another noteworthy phenomenon is that RGD and RSVRG tend to be less efficient than RCG. This demonstrates the advantage of using the Fisher information. 

\subsection{Deep Learning}
Batch normalization and momentum-based optimizer are standard techniques to train state-of-the-art image classification models \cite{he2016deep,sandler2018mobilenetv2,simonyan2014very}. 
We evaluate the proposed method with Kronecker-factorized approximate RFIM described in Section \ref{sec:practicalRNGD}, denoted by MKFAC, on VGG16BN \cite{simonyan2014very} and WRN-16-4 \cite{zagoruyko2016wide} while the benchmark datasets CIFAR-10/100 \cite{krizhevsky2009learning} are used. The detailed network structures are described in \cite{simonyan2014very, zagoruyko2016wide}.
%For brevity, we only choose two of the representative networks, VGG16BN and WRN-16-4 for comparisons. Here VGG16BN denotes that, batch normalization layers are added after each convolution or linear layer. 
In VGG16BN, batch normalization layers are added before every ReLU
activation layer. Additionally, we change the number of neurons in fully connected layers from
4096 to 512 and remove the middle layer of the last three in VGG due to memory
allocation problems (otherwise, one has to compute the inverse of $4096^2$-by-$4096^2$ matrices). This setting is also adopted in \cite{cho2017riemannian, yang2020sketchy}.
%which will not hamper the performance. 

The baseline algorithms are SGD, Adam, KFAC \cite{martens2015optimizing}, AdamP, and SGDP \cite{heo2020adamp}. The tangential projections are used to control the increase in norms of the weight parameters in AdamP and SGDP. These methods can be seen as approximate Riemannian first-order methods.
We fine tune the initial learning rates of the baseline algorithms by searching in the set $\{0.5, 0.2,0.1,0.05,0.02,0.01,\dots,5\times 10^{-5},2\times 10^{-5}, 10^{-5}\}.$ The learning rate decays in epoch 30, 60, and 90 with a decay rate $0.1$,  where an epoch is defined as one cycle through the full training dataset. We choose the parameters $\beta_1,\beta_2$ in Adam and AdamP from the set $\{0.9,0.99,0.999\}$. We search in the set $\{0.05,0.1,0.2,0.5,1,2\}$ to determine the damping parameter $\lambda$ used in calculating the natural direction $(F_k+ \lambda I)^{-1}g^k$ and update the KFAC matrix in epoch 30, 60, and 90. The initial damping parameter of KFAC is set to $2$ in all four tasks.
%on VGG16BN and $1$ on WRN-16-4, while they are multiplied by $0.5$ when the preconditioner updates at epoch 30, 60 and 90. 
We set the weight decay to $5 \times 10^{-4}$ for all algorithms. Each mini-batch contains 128 samples. The maximum number of epochs is set to $100$ for all algorithms. For MKFAC, we use RNGD for parameters constrained on the Grassmann manifold and SGD for the remaining parameters.
%different schemes according to whether the weights are on Grassmannian manifold. 
%For weights under batch normalization, as we have discussed in Section \ref{sec:3.2}, the Riemannian gradient coincides with the Euclidean gradient, thus KFAC is applied as a preconditioner and Riemannian stochastic gradient descent is applied. On the other hand, we use SGD for the weights without batch normalization. 
Let $\eta, \eta_g$ denote the learning rates for the Euclidean space and Grassmann
manifold, respectively. For the dataset CIFAR-10, we set $\eta_g=0.25$ and
$\eta=0.05$ with decay rates $0.2$ and $0.1$, respectively. The weight
decay is only applied to the unconstrained weights with parameter $5\times
10^{-4}$. The initial MKFAC damping parameters for WRN-16-4 and VGG16BN are set to 1 and 2 with decay rates $0.8$ and $0.5$, respectively, when the preconditioners update in epoch 30, 60, and 90. For the dataset CIFAR-100, we set $\eta_g=0.3$ for WRN-16-4,
$\eta_g=0.15$ for VGG16BN, and $\eta=0.05$ for both. The learning rate $\eta_g$ has a decay rate
$0.15$ for WRN-16-4 and $0.2$ for VGG16BN, while $\eta$ has a decay rate $0.1$ for
both of them. The initial MKFAC damping parameters for VGG16BN and WRN16-4 are set to 0.5 and 1 with decay rates $0.5$ and $0.8$, respectively. Other settings are the same as KFAC.
 \begin{table}[t]
	\centering
	\caption{Classification accuracy of various networks on CIFAR-10/100 (median of five runs).\label{tab:net}}
	\begin{tabular}{|c|c|c|c|c|}
		\hline Dataset &\multicolumn{2}{c|}{CIFAR-10}&\multicolumn{2}{c|}{CIFAR-100} \\
		\hline Model & WRN-16-4 & VGG16BN & WRN-16-4&VGG16BN \\
		\hline SGD &$93.84$	&$92.88$	&$74.30$	& $71.79$	\\
		\hline SGDP &$93.42$	&$92.49$	& $73.67$	& $71.54$	\\
		\hline Adam &	$92.53$&	 $89.88$&	$71.64$& $62.79$	\\
		\hline AdamP &	$92.55$ &$91.43$	&$71.23$	& $58.88$	\\
		\hline KFAC &$93.90$	&$94.36$&$74.31$	& $76.38$	\\
		\hline MKFAC & \textbf{94.06}	&\textbf{94.76}& \textbf{74.55}		&  \textbf{77.28}	\\
		\hline
	\end{tabular} 	
			
\end{table}

Table \ref{tab:net} presents the comparison of the baseline and the proposed
algorithms on CIFAR-10 and CIFAR-100 datasets. We list the best  classification
accuracy in 100 epochs, where the results are obtained from the median of 5 runs. The
performance of our proposed MKFAC method is the best in all four tasks. Compared with the second-order type method KFAC, our MKFAC method reaches higher accuracy, though KFAC has a much better behavior than SGD on these tasks. Compared with the manifold geometry-based first-order algorithms SGDP and AdamP, we see that using second-order information can give better accuracy than using first-order information alone.
%We see MKFAC obtain the minimal classification error rate among all algorithms. Especially on VGG networks, our proposed algorithm shows a state-of-the-art performance, even much better than KFAC
%One of the curves is shown in Figure   \ref{fig:net}.
%The baseline testing accuracy suffers from instability or experience a long plateau and cannot reach the summit, compared with the proposed methods. MKFAC tends to the most efficient in the early period and often gives the best classification rate with suitable adjusting of the learning rate. 
\section{Conclusion}
In this paper, we developed a novel efficient RNGD method for tackling the problem of minimizing a sum of negative log-probability losses over a manifold. Key to our development is a new notion of FIM on manifolds, which we introduced in this paper and could be of independent interest. We established the global convergence of RNGD and the local convergence rate of a deterministic version of RNGD. Our numerical results on representative machine learning applications demonstrate the efficiency and efficacy of the proposed method.

%Since this paper only deals with smooth objectives, designing efficient higher-order methods by exploring Fisher information for nonsmooth problems should be further studied.  

%\section*{Acknowledgements}
%
%We would like to thank xx for helping us ... Furthermore, we
%are very grateful to the referees for their most interesting comments
%and suggestions.

\bibliographystyle{siamplain}
\bibliography{ref}

\end{document}